\documentclass[a4paper,twoside,11pt,reqno]{amsart}
\usepackage{fullpage}
\usepackage[T1]{fontenc}
\usepackage[utf8]{inputenc}

\usepackage{hyperref}
\usepackage[slantedGreeks, partialup, noDcommand]{kpfonts}
\usepackage{graphicx}
\usepackage[mathscr]{euscript}

\hypersetup{ocgcolorlinks=true,allcolors=testc}
\hypersetup{
     colorlinks   = true,
     citecolor    = black
}
\hypersetup{linkcolor=black}
\hypersetup{urlcolor=black}

\newcommand{\eqdef}{\stackrel{\mathrm{def}}{=}}

\usepackage[dvipsnames]{xcolor}

\newtheorem{theorem}{Theorem}
\newtheorem{lemma}[theorem]{Lemma}
\newtheorem{question}[theorem]{Question}
\newtheorem{proposition}[theorem]{Proposition}

\newtheorem{remark}[theorem]{Remark}

\newtheorem*{definition*}{Definition}

\newcommand*\diff{\mathop{}\!\mathrm{d}}
\newcommand{\R}{\mathbb{R}} 

\newcommand{\N}{\mathbb{N}}

\newcommand{\li}{\bf }
\DeclareMathOperator{\arctanh}{arctanh}

\newcommand{\e}{\varepsilon}

\usepackage{dutchcal}

\newcommand{\MM}{\mathsf{M}}

\newcommand{\GG}{\mathsf{g}}

\newcommand{\mb}{\mathbb}
\newcommand{\ms}{\mathscr}
\newcommand{\msf}{\mathsf}
\newcommand{\mr}{\mathrm}

\newcommand{\II}{\mathsf{I}}
\newcommand{\III}{\mathscr{I}}
\newcommand{\HH}{\mathsf{H}} 
\newcommand{\DD}{\mathsf{D}} 

\title{Intrinsic dimensional functional inequalities on model spaces}

\author{Alexandros Eskenazis}
\address{(A.~E.) CNRS, Institut de Math\'ematiques de Jussieu, Sorbonne Universit\'e, France and Trinity College, University of Cambridge, UK.}
\email{alexandros.eskenazis@imj-prg.fr, ae466@cam.ac.uk}

\author{Yair Shenfeld}
\address{(Y.~S.)
Division of Applied Mathematics, Brown University, Providence, RI, USA}
\email{Yair\_Shenfeld@Brown.edu} 

\thanks{This material is based upon work supported by the NSF grant DMS-1929284 while A.~E.~was in residence at ICERM for the Harmonic Analysis and Convexity program. This material is based upon work supported by the National
Science Foundation under Award Number 2002022.}

\begin{document}

\maketitle
\vspace{-3mm}

\begin{abstract}
We initiate a systematic study of \emph{intrinsic} dimensional versions of classical functional inequalities which capture refined properties of the underlying objects. We focus on  model spaces: Euclidean space, Hamming cube, and manifolds of constant curvature. In the latter settings, our intrinsic dimensional functional inequalities improve on a series of known results and lead to new Hamilton-type matrix inequalities. Our proofs rely on scaling, tensorization, and stochastic methods.
\end{abstract}

\bigskip

{\footnotesize

}

\setcounter{tocdepth}{1}
\tableofcontents


\section{Introduction}
This work focuses on the development of \emph{intrinsic} dimensional versions of classical functional inequalities. In order to explain the meaning of ``intrinsic" in this context it is best to start with an important example. The logarithmic Sobolev inequality in Gauss space \cite{Sta59, Gro75} asserts that for every nice-enough absolutely continuous probability measure $\mu$ on $\R^n$,
\begin{equation} \label{eq:lsi-gross}
\HH(\mu\|\gamma_n) \leq\frac{1}{2} \II(\mu\|\gamma_n),
\end{equation}
where $\gamma_n$ is the standard Gaussian measure on $\R^n$. Here,
\begin{equation}
\HH(\mu\|\nu) \eqdef \int \log\Big(\frac{\diff \mu}{\diff \nu}\Big) \,\diff \mu 
\end{equation}
is the \emph{relative entropy} of $\mu$ with respect to $\nu$ and 
\begin{equation}
\II(\mu\|\nu) \eqdef \int \Big| \nabla\log\frac{\diff\mu}{\diff\nu} \Big|^2 \,\diff\mu = \int \frac{|\nabla (\diff\mu/\diff\nu)|^2}{\diff\mu/\diff\nu} \,\diff\nu
\end{equation}
is the \emph{relative Fisher information} of $\mu$ with respect to $\nu$, provided that $\mu<\!\!\!<\nu$.

Gross' motivation for \eqref{eq:lsi-gross} was to find a substitute for the Euclidean Sobolev inequalities which holds in infinite-dimensional spaces (which was needed in constructive quantum field theories). Sobolev inequalities have the feature that the dimension $n$ of the ambient space $\R^n$ appears explicitly in the constants of the inequalities, which leads to triviality upon taking the limit $n\to\infty$. In contrast, the constant $1/2$ appearing in \eqref{eq:lsi-gross} is \emph{dimension-free}, leading to \eqref{eq:lsi-gross} being well-defined in infinite dimensions. On the other hand, as was already observed by Stam \cite{Sta59}, \eqref{eq:lsi-gross} can in fact be improved if the dimension $n$ is taken into account. To see this improvement we first apply a standard change of measure (see \cite{Wei78}) which shows that \eqref{eq:lsi-gross} is equivalent to
\begin{equation} \label{eq:lsi}
\HH(\mu\|\lambda_n) - \HH(\gamma_n\|\lambda_n) \leq \frac{\II(\mu\|\lambda_n)-n}{2},
\end{equation}
where $\lambda_n$ is the Lebesgue measure on $\R^n$. The \emph{dimensional} log-Sobolev inequality \cite{Sta59, CC84, Car91},
\begin{equation} \label{eq:lsi_dim}
\HH(\mu\|\lambda_n) - \HH(\gamma_n\|\lambda_n) \leq \frac{n}{2}\log\left(\frac{\II(\mu\|\lambda_n)}{n}\right),
\end{equation}
improves upon \eqref{eq:lsi} as can be seen from the inequality $\log s\le s-1$ for $s\in (0,\infty)$.  It is clear that when the Fisher information is large, \eqref{eq:lsi_dim} provides an exponential refinement over \eqref{eq:lsi}. Despite this quantitative improvement,  \eqref{eq:lsi_dim} suffers from a lack of sensitivity to the \emph{intrinsic} dimension of $\mu$.  To see this, suppose that $\mu$ is of the form $\diff\mu(x_1,\ldots, x_n)=\diff\tilde\mu(x_1,\ldots,x_k)\diff\gamma_{n-k}(x_{k+1},\ldots,x_n)$, where $k<n$ and $\tilde \mu$ is an absolutely continuous probability measure on $\R^k$.  Then \eqref{eq:lsi_dim} rephrased in terms of $\tilde{\mu}$ asserts that
\begin{equation}
\HH(\tilde \mu\|\lambda_k) - \HH(\gamma_k\|\lambda_k) \leq \frac{n}{2}\log\left(1+ \frac{\II(\tilde \mu\|\lambda_k) - k}{n}\right),
\end{equation}
which deteriorates to \eqref{eq:lsi} as the \emph{ambient} dimension $n$ increases, despite the fact that the intrinsic dimension $k$ of $\mu$ is fixed.  In other words, \eqref{eq:lsi_dim} is insensitive to the structure of $\mu$.  In \cite[p.~12]{Dem90}, Dembo showed that \eqref{eq:lsi_dim} can be further improved to an inequality which captures the intrinsic dimension of $\mu$:
\begin{equation} \label{eq:lsi**}
\HH(\mu\|\lambda_n) - \HH(\gamma_n\|\lambda_n) \leq \frac{1}{2}\log\det \III(\mu\|\lambda_n),
\end{equation}
where
\begin{equation}
\III(\mu\|\nu) \eqdef \int\frac{(\nabla (\diff\mu/\diff\nu))^{\otimes2}}{\diff\mu/\diff\nu}\,\diff\nu
\end{equation}
is the {\it relative Fisher information matrix} of $\mu$ with respect to $\nu$.  Observe that
\begin{equation}
\II(\mu\|\nu) = \mr{tr}\, \III(\mu\|\nu),
\end{equation}
and thus \eqref{eq:lsi**} improves on \eqref{eq:lsi_dim} by the elementary inequality $\log\det C \leq n\log \tfrac{\mr{tr}\, C}{n}$ which holds for every $n\times n$ positive semidefinite matrix $C$.  In particular, both sides of  \eqref{eq:lsi**} behave additively with respect to product measures: Plugging in $\diff\mu=\diff\tilde\mu\diff\gamma_{n-k}$ into \eqref{eq:lsi**} yields
\begin{equation} \label{eq:lsi**k}
\HH(\tilde\mu\|\lambda_k) - \HH(\gamma_k\|\lambda_k) \leq \frac{1}{2}\log\det \III(\tilde\mu\|\lambda_k)
\end{equation}
which captures correctly the intrinsic dimension of $\mu$.
More generally, by considering the eigenvalues of the Fisher information matrix, \eqref{eq:lsi**}  can quantify the extent to which $\mu$  degenerates along each eigenvector direction.

The goal of this work is to initiate a systematic study of intrinsic dimensional versions of classical functional inequalities. We focus on some important model spaces: Euclidean space, Hamming cube, and  space forms (manifolds of constant sectional curvature). These model spaces have historically played a crucial role in the development of functional inequalities and their study has been the impetus leading to fruitful generalizations and abstractions; see the monograph \cite{BGL14}.  In view of the richness of the subject, our intrinsic dimensional functional inequalities on these spaces improve on multiple classical inequalities from the literature. The tools required to establish intrinsic dimensional functional inequalities in each of the model spaces will depend on the unique characteristics of the space itself: {\bf scaling} (Euclidean space), {\bf tensorization} (Hamming cube), and {\bf stochastic methods} (space forms). In the rest of the introduction we will review each of these methods and present examples of the intrinsic dimensional functional inequalities which follow. We defer the statements of many of our results to the main body of the paper; see the following brief summary: 

\medskip

{\bf Part 1. Euclidean and product spaces: scaling and tensorization}
\begin{itemize}
\item Logarithmic Sobolev inequalities for homogeneous measures (Section \ref{subsec:homogeneous}).
\item  Bayesian Cram\'er--Rao bounds (Section \ref{subsec:CR_bound}).
\item Gagliardo–Nirenberg–Sobolev inequalities (Section \ref{sec:gns}).
\item Beckner inequalities (Section \ref{sec:beckner}).
\item $q$-logarithmic Sobolev inequalities (Section \ref{sec:qlsi}).
\item Nonlinear logarithmic Sobolev inequalities in product spaces (Section \ref{sec:Hamming}).
\end{itemize}

{\bf Part 2. Space forms: stochastic methods}
\begin{itemize}
\item Local logarithmic Sobolev inequalities on space forms (Section \ref{sec:LSImanifold}).
\item Local logarithmic Sobolev inequalities and Hamilton's matrix inequalities on nonpositively curved space forms (Section \ref{sec:nonpos}).
\end{itemize}


\subsection{Euclidean spaces: scaling} \label{subsec:scaling}
Most classical functional inequalities on $\R^n$ are coordinate-free results phrased in a coordinate-dependent way. As such, they can often be substantially refined when expressed in a suitable basis. Concretely, the correct basis is found by performing a change of variables of the form $x\mapsto Ax$ and then optimizing over a prescribed class of symmetries $A\in\msf{G}\subseteq \msf{GL}_n$.  Let us remark that explicit improvements of this form can be obtained only when it is possible to solve these optimization problems, which is not always the case. These improvements are moreover motivated by the study of equality cases. When a functional inequality has a non-constant function $h:\R^n\to\R$ as an equality case, then the refined inequality obtained in the manner described above would be saturated by all functions of the form $h_A(x) = h(Ax)$, where $A\in\msf{G}$. This principle has already been applied by Dembo \cite{Dem90} in the case of the Gaussian logarithmic Sobolev inequality (see also \cite{ELS20,BNT22} and Section \ref{sec:lsi} below). In the first part of the paper we shall present more applications of this idea to other important functional inequalities in Euclidean space and further consequences.

\subsubsection{Beckner inequalities} 
In \cite{Bec89}, Beckner proved that any smooth function $u\in C_0^\infty(\R^n)$ satisfies the estimates
\begin{equation} \label{eq:beckner}
\forall \ p\in[1,2), \qquad \|u\|_{L_2(\gamma_n)}^2-\|u\|_{L_p(\gamma_n)}^2 \leq (2-p) \|\nabla u\|_{L_2(\gamma_n)}^2.
\end{equation}
This family of inequalities interpolates between the Gaussian Poincar\'e inequality (corresponding to $p=1$) and Gross' logarithmic Sobolev inequality \cite{Gro75} which arises as a limit when $p\to2^-$. We refer to the influential work of Lata\l a and Oleszkiewicz \cite{LO00} as well as \cite[Section~7.6]{BGL14} for examples of Beckner-type inequalities satisfied by non-Gaussian measures. 

In \cite[Corollary~4]{DT16}, Dolbeault and Toscani proposed a \emph{dimensional} refinement of Beckner's inequality \eqref{eq:beckner} for functions satisfying a second moment normalization condition. More specifically, they showed that if a function $u\in C_0^\infty(\R^n)$ satisfies the normalization condition
\begin{equation}
\int_{\R^n} |x|^2 u(x)^2\,\diff \gamma_n(x) = n \|u\|_{L_2(\gamma_n)}^2,
\end{equation}
then
\begin{equation} \label{eq:dt}
\forall p\in[1,2),\qquad \|u\|_{L_2(\gamma_n)}^2 \varphi_{p,n}\left(1-\frac{\|u\|_{L_p(\gamma_n)}^2}{\|u\|_{L_2(\gamma_n)}^2}\right) \leq \|\nabla u\|_{L_2(\gamma_n)}^2,
\end{equation}
where the function $\varphi_{p,n}$ is given by
\begin{equation}
\forall \ s\in(0,1),\qquad \varphi_{p,n}(s) \eqdef \frac{n}{4} \Big( (1-s)^{-\frac{2p}{n(2-p)}}-1\Big).
\end{equation}
Observe that \eqref{eq:dt} improves upon \eqref{eq:beckner} up to the value of the implicit constant as
\begin{equation}
\forall \ s\in(0,1),\qquad \varphi_{p,n}(s) \geq \frac{p}{2(2-p)} \log\Big(\frac{1}{1-s}\Big) \geq \frac{p}{2(2-p)}\cdot s.
\end{equation}
The improvement \eqref{eq:dt} becomes particularly substantial when $\|u\|_{L_p(\gamma_n)}<\!\!\!<\|u\|_{L_2(\gamma_n)}$.

In the spirit of the matricial refinement \eqref{eq:lsi**} over the dimensional logarithmic Sobolev inequality \eqref{eq:lsi_dim}, we present the following refinement of \eqref{eq:dt} for functions whose second moment \emph{matrix} is appropriately normalized.

\begin{theorem} \label{thm:beckner}
Fix $n\in\N$ and let $u\in C_0^\infty(\R^n)$ be such that
\begin{equation}
\forall \ i,j\in\{1,\ldots,n\},\qquad \int_{\R^n} x_i x_j u(x)^2 \,\diff\gamma_n(x) = \delta_{ij} \|u\|_{L_2(\gamma_n)}^2,
\end{equation}
where $\delta_{ij}$ is the Kronecker delta. Then, we have
\begin{equation} \label{eq:dt*}
\forall \ p\in[1,2),\qquad \frac{\|u\|_{L_2(\gamma_n)}^2-\|u\|_{L_p(\gamma_n)}^2}{\|u\|_{L_2(\gamma_n)}^2} \leq 1-\left[\det\left(\frac{4}{\|u\|_{L_2(\gamma_n)}^2} \int_{\R^n} (\nabla u)^{\otimes 2} \,\diff\gamma_n + \msf{Id}_n\right)\right]^{-\frac{2-p}{2p}}.
\end{equation}
\end{theorem}

Applying the inequality $\det C \leq \left(\tfrac{\mr{tr} C}{n}\right)^n$ and rearranging, we see that \eqref{eq:dt*} strengthens \eqref{eq:dt}.

\subsubsection{Gagliardo--Nirenberg--Sobolev inequalities}
Fix $n\in\N$. The Gagliardo--Nirenberg inequality \cite{Gag59,Nir59} asserts that for every $p,q,r,s\in[1,\infty)$ and $\theta\in[0,1]$ satisfying the constraint
\begin{equation} \label{eq:gn_intro}
\frac{1}{p} = \frac{\theta}{q}+ \Big(\frac{1}{r}-\frac{1}{n}\Big) (1-\theta),
\end{equation}  
there exists a universal (optimal) constant $C^{p,q,r,s}>0$ such that every $u\in C_0^\infty(\R^n)$ satisfies
\begin{equation} \label{eq:gns-gen_intro}
\|u\|_{L_p(\R^n)} \leq C^{p,q,r,s}\|u\|_{L_q(\R^n)}^\theta \|\nabla u\|_{L_r(\R^n;\ell_s^n)}^{1-\theta},
\end{equation}
where we use the standard notation
\begin{equation}\label{eq:nabla_r_notation_intro}
\|\nabla u\|_{L_r(\R^n;\ell_s^n)} = \left( \int_{\R^n} \Big( \sum_{i=1}^n |\partial_i u(x)|^s \Big)^{r/s} \,\diff x\right)^{1/r}.
\end{equation}
In the special case $r\in(1,n)$ and $\theta=0$, inequality \eqref{eq:gns-gen_intro} boils down to the classical Sobolev inequality \cite{Sob38, Sob63}. The endpoint case $r=1$ and $\theta=0$ was due to \cite{Gag59,Nir59} and the corresponding optimal constant for $s=2$ was found by Federer, Fleming and Rishel \cite{FF60, FR60}. The optimal constant in the range $r\in(1,n)$ and $\theta=0$ for $s=2$ was discovered by Aubin and Talenti \cite{Aub76,Tal76}.  The logarithmic Sobolev inequality \eqref{eq:lsi} can be obtained as an endpoint case of the Gagliardo--Nirenberg--Sobolev inequality \eqref{eq:gns-gen_intro} with the optimal constant when $s=2$ (see \cite[Section~1]{dPD03}).  Finally,  the optimal constant $C^{p,q,r,s}$ for general parameters was found by Cordero-Erausquin, Nazaret and Villani in \cite[Section~3]{CNV04}.  In this paper, we present a refined inequality for $r=s$.

\begin{theorem} \label{thm:gns_intro}
Let $p,q,r\in[1,\infty)$, $\theta\in[0,1]$ and $C^{p,q,r,r}>0$ be such that \eqref{eq:gns-gen_intro} is satisfied for all functions $u\in C_0^\infty(\R^n)$ with $r=s$ under the constraint \eqref{eq:gn_intro}. Then, for every $u\in C_0^\infty(\R^n)$, we have
\begin{equation} \label{eq:gns*_intro}
\|u\|_{L_p(\R^n)} \leq C^{p,q,r,r} n^{\frac{1-\theta}{r}} \|u\|_{L_q(\R^n)}^\theta \Big( \prod_{j=1}^n \|\partial_ju\|_{L_r(\R^n)}\Big)^{\frac{1-\theta}{n}}.
\end{equation}
\end{theorem}
The inequality \eqref{eq:gns*_intro} improves on \eqref{eq:gns-gen_intro} by the arithmetic mean-geometric mean inequality so Theorem \ref{thm:gns_intro} asserts that Euclidean Gagliardo--Nirenberg--Sobolev inequalities, that is, inequalities of the form \eqref{eq:gns-gen_intro} with the choice of parameter $r=s$, self-improve via scaling. In particular, \eqref{eq:gns*_intro} captures the fact (absent from \eqref{eq:gns-gen_intro}) that $\partial_i u \equiv 0$ on $\R^n$ implies that $u\equiv 0$ under any $L_s$-integrability assumption for $u$.

\subsection{Product spaces: tensorization} \label{sec:1.3}
If $(\Omega,\pi)$ is a probability space, then for a measurable function $f:\Omega\to\R_+$ we shall denote its entropy with respect to $\pi$ by
\begin{equation}
\mathrm{Ent}_\pi[f] \eqdef \int_\Omega f \log f\,\diff\pi - \Big( \int_\Omega f \,\diff\pi\Big) \, \log\Big( \int_\Omega f\,\diff\pi\Big).
\end{equation}The usefulness of logarithmic Sobolev inequalities in probability and geometry stems largely from the fact that entropy satisfies a simple yet powerful tensorization principle, rendering them dimension-free estimates \cite{L99}.  In the interesting work \cite{PS19},  Polyanskiy and Samorodnitsky introduced a family of more general inequalities for Markov semigroups called \emph{nonlinear} logarithmic Sobolev inequalities (see also \cite{Gro75, Wei78, CC84, DS84, Car91, Mic99,Sam08} for previous occurrences of such estimates in the literature and applications).  Let $\{P_t\}_{t\geq0}$ be a Markov semigroup acting on measurable functions $f:\Omega\to\R$  with stationary measure $\pi$. Following \cite{PS19}, we say that $\{P_t\}_{t\geq0}$ satisfies the $(p,\Phi)$--LSI, where $p\geq1$ and $\Phi:\R_+\to\R_+$ is a concave, continuous function with $\Phi(0)=0$,  if for every measurable function $f:\Omega\to\R_+$, we have
\begin{equation}
\mr{Ent}_\pi[f^p] \leq \mb{E}_\pi[f^p] \Phi\left(\frac{\ms{E}(f,f^{p-1})}{\mb{E}_\pi[f^p]}\right),
\end{equation}
where $\ms{E}(\cdot,\cdot)$ is the Dirichlet form corresponding to $\{P_t\}_{t\geq0}$. As usual, the term $\ms{E}(f,f^{p-1})$ is interpreted as $\ms{E}(f,\log f)$ in the endpoint case $p=1$.

In \cite[Theorem~1]{PS19}, the authors proved a dimensional tensorization property for nonlinear log-Sobolev inequalities asserting that if $\{P_t\}_{t\geq0}$ satisfies the $(p,\Phi)$--LSI, then for any $n\geq1$, the product semigroup $\{P_t^{\otimes n}\}_{t\geq0}$ with stationary measure $\pi^n$ satisfies the $\big(p, n\Phi(\tfrac{1}{n}\cdot )\big)$--LSI:
\begin{equation} \label{eq:into_PS}
\mathrm{Ent}_{\pi^n}[f^p] \leq  n\mb{E}_{\pi^n}[f^p]\Phi\left( \frac{\ms{E}(f,f^{p-1})}{n\mb{E}_{\pi^n}[f^p]}\right).
\end{equation}
By considering functions $f$ of the form $f(x_1,\ldots,x_n)= \tilde f(x_1,\ldots, x_k)$, for $k<n$, we see that \eqref{eq:into_PS} suffers from the problem of incorporating the ambient dimension $n$ into the constant, thus ignoring the structure of $f$. In the Euclidean setting, we overcame this issue by finding the correct basis via an optimization procedure over the cone of positive semidefinite matrices. In contrast, such an approach is not suitable on the Hamming cube due to its discrete nature. Our solution to this problem is to refine tensorization instead of scaling. Indeed, as a consequence of a more general tensorization principle (see Theorem \ref{thm:tensor} below), we shall prove the following stronger nonlinear logarithmic Sobolev inequality for product spaces.

\begin{theorem} \label{thm:tensor-intro}
Let $(\Omega,\pi,\{P_t\}_{t\geq0})$ be a stationary Markov semigroup satisfying the $(p,\Phi)$--LSI for some $p\geq1$ and some concave, continuous function $\Phi:\R_+\to\R_+$ with $\Phi(0)=0$. Then, for any $n\geq1$, every measurable function $f:\Omega^n\to\R_+$ satisfies
\begin{equation} \label{eq:tensor-intro-conclusion}
\mathrm{Ent}_{\pi^n}[f^p] \leq \mb{E}_{\pi^n}[f^p] \sum_{i=1}^n \Phi\left( \frac{\mb{E}_{\pi^n}\big[\ms{E}_i(f,f^{p-1})\big]}{\mb{E}_{\pi^n}[f^p]}\right),
\end{equation}
where $\ms{E}_i(\cdot,\cdot)$ is the Dirichlet form associated with the $i$-th component of the semigroup $\{P_t^{\otimes n}\}_{t\geq0}$.
\end{theorem}

It follows readily from Jensen's inequality that
\begin{equation}
 \sum_{i=1}^n \Phi\left( \frac{\mb{E}_{\pi^n}\big[\ms{E}_i(f,f^{p-1})\big]}{\mb{E}_{\pi^n}[f^p]}\right) \leq  n\Phi\left( \frac{\ms{E}(f,f^{p-1})}{n\mb{E}_{\pi^n}[f^p]}\right),
\end{equation}
where $\ms{E}(\cdot,\cdot)$ is the Dirichlet form associated to $\{P_t^{\otimes n}\}_{t\geq0}$ and thus \eqref{eq:tensor-intro-conclusion} indeed strengthens \eqref{eq:into_PS}. Moreover, in \cite[Theorems~4 and 6]{PS19}, the authors found the \emph{optimal} functions $\Phi_p$ such that the $(p,\Phi_p)$-LSI is satisfied on the one-dimensional Hamming cube $\{0,1\}$ equipped with the uniform measure. Tensorizing their result via Theorem \ref{thm:tensor-intro}, one deduces an improved nonlinear logarithmic Sobolev inequality on the Hamming cube $\{0,1\}^n$.  


\subsection{Space forms: stochastic methods} \label{sec:intro_space} In order to explain our intrinsic dimensional functional inequalities on space forms we first recall the notion of \emph{local} logarithmic Sobolev inequalities. Starting with the Euclidean setting, fix $T\ge 0$, $x\in \R^n$, and let $\frac{\diff\mu}{\diff \lambda_n}=\frac{fP_T\delta_x}{P_Tf(x)}$ where $\delta_x$ is the Dirac mass at $x$, $f:\R^n\to \R$ is a nonnegative function, and $\{P_t\}_{t\geq0}$ is the Euclidean heat semigroup given by $P_th(x):=\int h(x+\sqrt{t}z)d\gamma_n(z)$. Plugging $\mu$ into \eqref{eq:lsi_dim} yields (after integration by parts and using the explicit form of $P_T\delta_x$),
\begin{align}
\label{eq:local_lsi_dim}
P_T(f\log f)(x)-P_Tf(x)\log P_Tf(x)\le \frac{T}{2}P_T\Delta f (x)+\frac{n}{2}P_Tf(x)\log\left(1-\frac{T}{n}\frac{P_T(f\Delta\log f)(x)}{P_Tf(x)}\right).
\end{align}
The inequality \eqref{eq:local_lsi_dim} is the \emph{local} dimensional logarithmic Sobolev inequality on $\R^n$ \cite{BL06}. While \eqref{eq:local_lsi_dim} provides an \emph{upper} bound on the (local) entropy, the  \emph{reverse} local dimensional logarithmic Sobolev inequality \cite{BL06} provides a \emph{lower} bound,
\begin{align}
\label{eq:local_lsi_dim_reverse}
P_T(f\log f)(x)-P_Tf(x)\log P_Tf(x)\ge \frac{T}{2}P_T\Delta f (x)-\frac{n}{2}P_Tf(x)\log\left(1+\frac{T}{n}\Delta\log P_Tf(x)\right).
\end{align}
Analogously, we can use \eqref{eq:lsi**}, instead of \eqref{eq:lsi_dim}, to get the local intrinsic dimensional logarithmic Sobolev inequality on $\R^n$,
\begin{align}
\label{eq:intrinisc_lsi_dim}
P_T(f\log f)(x)-P_Tf(x)\log P_Tf(x)\le \frac{T}{2}P_T \Delta f (x)+\frac{1}{2}P_Tf(x)\log\det\left(\msf{Id}_n -T\frac{P_T(f\nabla^2\log f(x))}{P_Tf(x)}\right),
\end{align}
which improves on \eqref{eq:local_lsi_dim}. As for a \emph{reverse}  local intrinsic dimensional logarithmic Sobolev inequality in $\R^n$, we will establish below (Theorem \ref{thm:flat}) that
\begin{align}
\label{eq:intrinisc_lsi_dim_reverse}
P_T(f\log f)(x)-P_Tf(x)\log P_Tf(x)\ge \frac{T}{2}\Delta P_Tf(x)-\frac{1}{2}P_Tf(x)\log\det\left(\msf{Id}_n+T\nabla^2\log P_Tf(x)\right),
\end{align}
which improves on \eqref{eq:local_lsi_dim_reverse}.

Turning to the manifold setting, local dimensional logarithmic Sobolev inequalities exist on manifolds in forms which account for both the dimension of the manifold as well as the Ricci curvature \cite{BBL17}. In light of the existence of the local intrinsic dimensional logarithmic Sobolev inequalities on Euclidean spaces \eqref{eq:intrinisc_lsi_dim} and \eqref{eq:intrinisc_lsi_dim_reverse}, we wish to understand whether such inequalities can also exist on manifolds. Upon closer inspection, however, it is clear that inequalities such as \eqref{eq:intrinisc_lsi_dim} and \eqref{eq:intrinisc_lsi_dim_reverse} cannot hold if the only curvature information given pertains to the Ricci tensor. On a conceptual level, the difference between the dimensional and intrinsic dimensional inequalities is that the former provide information about the \emph{trace} of the Fisher information matrix, while the latter provide information about the \emph{full spectrum}. Hence,  while information on the trace of the Riemann tensor, i.e., Ricci curvature,  suffices to yield a dimensional inequality, information on the full Riemann tensor, i.e., sectional curvature,  should be required to give an intrinsic dimensional inequality.

A concrete manifestation of this intuition is exhibited by the inequalities of Li--Yau and Hamilton \cite{LY86,Ham93}. As was realized in \cite{BL06}, the reverse  local dimensional logarithmic Sobolev inequality \eqref{eq:local_lsi_dim_reverse} implicitly implies the Li--Yau inequality on $\R^n$,
\begin{equation}\label{eq:Li-Yau}
\forall x\in \R^n, \qquad -\Delta \log P_Tf(x)\le \frac{n}{T},
\end{equation}
since the argument in the log term of \eqref{eq:local_lsi_dim_reverse} must be nonnegative. Analogously, the reverse  local intrinsic dimensional logarithmic Sobolev inequality \eqref{eq:intrinisc_lsi_dim_reverse} implies Hamilton's inequality,
\begin{equation}\label{eq:Hamilton}
\forall x\in \R^n, \qquad-\nabla^2 \log P_Tf(x) \preceq\frac{1}{T}\msf{Id}_n,
\end{equation}
where $\preceq$ is the order of positive semidefinite matrices.  In the manifold setting, the  Li--Yau inequality, which is a statement about the \emph{trace} of the Hessian of $\log P_Tf$, holds under a nonnegativity assumption on the \emph{trace} of the Riemann tensor, namely the Ricci tensor \cite{LY86,YZ20}. Indeed,  Bakry and Ledoux \cite{BL06} (see also the follow-up work \cite{BBL17}) established (reverse) local dimensional logarithmic Sobolev inequalities on manifolds with nonnegative \emph{Ricci} curvature which imply the Li--Yau inequality. In contrast, Hamilton's inequality, which is a statement about the \emph{Hessian} of $\log P_Tf$, requires the manifold to have nonnegative sectional curvature (and also to be  Einstein), which is an assumption on the \emph{full spectrum} of the Riemann tensor \cite{Ham93}. It follows that if local intrinsic dimensional logarithmic Sobolev inequalities were to hold, then information about the sectional curvature should be provided.

In this work we establish local intrinsic dimensional logarithmic Sobolev inequalities as well as Hamilton-type matrix inequalities for space forms: Euclidean spaces, spheres, and hyperbolic spaces. In addition to serving as the model spaces for functional inequalities on manifolds, these spaces are the simplest non-trivial examples of manifolds  where we could hope for local intrinsic dimensional logarithmic Sobolev inequalities to hold. The methods of scaling and tensorization which worked, respectively, for Euclidean spaces and product spaces no longer apply on curved spaces as they lack product and homogeneity structures. Hence, we take a different route and build on the stochastic approach of Lehec  \cite{Leh13,Leh17} and Eldan, Lehec, and Shenfeld \cite{ELS20} towards logarithmic Sobolev inequalities. We start by stating our local intrinsic dimensional logarithmic Sobolev on space forms while deferring precise definitions to Part \ref{part:space_forms}.

\begin{theorem} \label{thm:manifolds}
Let $(\MM,\GG)$ be an $n$-dimensional Riemannian manifold with constant sectional curvature $\kappa\in\R\setminus\{0\}$ with the associated  heat semigroup $\{P_t\}_{t\ge 0}$. Fix $T>0$, $x\in\MM$, a smooth positive function $f:\MM\to\R$ with $\int_M f\,\diff P_T\delta_x=1$, and let $\mu$ be the probability measure with $\frac{\diff\mu}{\diff P_T\delta_x}=f$. Define the 2-tensor $C(t) = \frac{e^{n\kappa t}}{n\kappa}A+tB$ for $t\in\R$ where $A, B$ are the 2-tensors given by
\begin{equation}
\begin{cases}
A =- e^{-n\kappa T}  \big(P_T\nabla^2 f(x) - \frac{1}{n} P_T \Delta f(x)\cdot \GG\big) \\ B = \Big(\frac{(n-1)\kappa}{2} - \frac{\Delta P_T f(x)}{n}\Big) \cdot \GG.
\end{cases}
\end{equation}
Then, we have the local intrinsic dimensional logarithmic Sobolev inequality
\begin{align} \label{eq:manifolds-lsi}
\begin{split}
&P_T(f\log f)(x)-P_Tf(x)\log P_Tf(x)\\
& \leq \frac{1}{2} \int_0^T\!\!\!  \mr{tr}\Big[e^{C(t)-C(T)} \Big(\GG + \mb{E}_\mu \big(\nabla \log f\big)^{\otimes 2}\!\! \int_t^T e^{2C(s)-2C(T)} \,\diff s \Big)^{-1}\!\! \mb{E}_\mu \big(\nabla \log f\big)^{\otimes 2}  e^{C(t)-C(T)}\Big]\,\diff t,
\end{split}
\end{align}
and the reverse local intrinsic dimensional logarithmic Sobolev inequality
\begin{align} \label{eq:manifolds-rev-lsi}
\begin{split}
&P_T(f\log f)(x)-P_Tf(x)\log P_Tf(x)\\
& \geq \frac{1}{2} \int_0^T \mr{tr}\Big[e^{C(t)-C(0)} \Big( \GG -\big( \nabla \log P_Tf(x)\big)^{\otimes 2} \int_0^t e^{2C(s)-2C(0)}\,\diff s\Big)^{-1} \big( \nabla \log P_Tf(x)\big)^{\otimes 2} e^{C(t)-C(0)}\Big] \,\diff t.
\end{split}
\end{align}
\end{theorem}
As will become clear from the proof of Theorem \ref{thm:manifolds}, the theorem is not optimal and follows from a more powerful ``master " matrix differential inequality (section \ref{subsec:discussion}). There are other inequalities which can be deduced from the master matrix differential inequality, specifically in space forms with nonpositive sectional curvature. In particular, we prove Hamilton-type matrix inequalities for the heat equation:
\begin{theorem}
\label{thm:Hamilton_intro}
 Let $(\MM,\GG)$ be an $n$-dimensional Riemannian manifold with constant nonpositive sectional curvature $\kappa \le 0$. Let $\{P_t\}_{t\ge 0}$ be the associated heat semigroup and let  $f:\MM\to \R$  be a positive function. Then, for every $x\in\MM$ and every $T\ge 0$, 
\begin{align}
\label{eq:Hamilton_matrix_inq_negcurved_lambda=0_intro}
\begin{split}
&\text{if, either }\kappa=0, \text{ or }\kappa<0\text{ and }\frac{4}{n^2\kappa}\frac{\Delta P_Tf(x)}{P_Tf(x)}=1,\quad\text{then}\quad -\nabla^2\log P_Tf(x)	\preceq   \frac{1}{T}\msf{Id}_n\quad\forall x\in \MM.
\end{split}
\end{align}
Further,
\begin{align}
\label{eq:Hamilton_matrix_inq_negcurved_lambda>0_intro}
\begin{split}
&\text{if }\kappa<0\text{ and }\frac{4}{n^2\kappa}\frac{\Delta P_Tf(x)}{P_Tf(x)}>1,\\
&\text{then}\quad 
-\nabla^2\log P_Tf(x)	\preceq \frac{n\kappa}{2}\left\{\sqrt{\frac{4}{n^2\kappa}\frac{\Delta P_Tf(x)}{P_Tf(x)}-1}\cot\left(\frac{n\kappa T}{2}\sqrt{\frac{4}{n^2\kappa}\frac{\Delta P_Tf(x)}{P_Tf(x)}-1}\right)-1\right\}\msf{Id}_n.
\end{split}
\end{align}
\end{theorem}
In flat space, where $\kappa=0$, Theorem \ref{thm:Hamilton_intro} reduces to \eqref{eq:Hamilton}, namely, Hamilton's matrix inequality \cite[Corollary 4.4]{Ham93}. In hyperbolic spaces, Theorem \ref{thm:Hamilton_intro} is completely new. The constraint $\frac{4}{n^2\kappa}\frac{\Delta P_Tf(x)}{P_Tf(x)}>1$ is natural. Indeed, Theorem \ref{thm:Hamilton_intro} is a matrix version of the improved Li-Yau inequality of Bakry, Bolley, and Gentil---see Remark \ref{rem:BBG}.

Going beyond matrix inequalities, we can use our master matrix differential inequality to obtain another form of local intrinsic dimensional logarithmic Sobolev inequalities.
\begin{theorem}
\label{thm:nge_curved_intro}
Let $(\MM,\GG)$ be the $n$-dimensional hyperbolic space with sectional curvature $\kappa<0$ with the associated  heat semigroup $\{P_t\}_{t\ge 0}$. Fix $T>0$, $x\in\MM$, a smooth positive function $f:\MM\to\R$ with $\int_M f\,\diff P_T\delta_x=1$,
and let $\mu$ be the probability measure with $\frac{\diff\mu}{\diff P_T\delta_x}=f$. 
Then, with
\begin{align}
\lambda\eqdef \frac{n^2\kappa^2}{4}\left\{\frac{4}{n^2\kappa}\Delta P_Tf(x)-1\right\}, \quad 
\alpha_i\eqdef
\begin{cases}
\arctan\left(\frac{1}{\sqrt{\lambda}}\left(\sigma_i+\frac{n\kappa}{2}\right)\right)  &\text{if }\lambda>0,\\
-\frac{2}{2\sigma_i+n\kappa}  &\text{if }\lambda=0,\\
\arctanh\left(-\frac{1}{\sqrt{-\lambda}}\left(\sigma_i+\frac{n\kappa}{2}\right)\right)&\text{if }\lambda<0,\\
\end{cases}
\end{align}
we have the local intrinsic dimensional logarithmic Sobolev inequality
\begin{align} \label{eq:manifolds-lsi_nonpos}
\begin{split}
P_T(f\log f)(x)-P_Tf(x)\log P_Tf(x)\leq \frac{ P_T\Delta f(x)}{2}-\frac{n^2\kappa T}{2}-\frac{1}{2}
\begin{cases}
\sum_{i=1}^n\log\left(\frac{\cos (\alpha_i)}{\cos(\sqrt{\lambda}T+\alpha_i)}\right)&\text{if }\lambda>0\\
\sum_{i=1}^n\log\left(\frac{\alpha_i}{T+\alpha_i}\right) &\text{if }\lambda=0\\
\sum_{i=1}^n\log\left(\frac{\cosh (\alpha_i)}{\cosh(\sqrt{-\lambda}T+\alpha_i)}\right) &\text{if }\lambda<0\\
\end{cases}
\end{split}
\end{align}
where $\{\sigma_i\}_{i=1}^n$ are the eigenvalues of $\mb{E}_{\mu}[-\nabla^2\log f]$, and the reverse local intrinsic dimensional logarithmic Sobolev inequality
\begin{align} \label{eq:manifolds-rev-lsi_nonpos}
\begin{split}
P_T(f\log f)(x)-P_Tf(x)\log P_Tf(x)\geq \frac{ P_T\Delta f(x)}{2}-\frac{n^2\kappa T}{2}+\frac{1}{2}
\begin{cases}
\sum_{i=1}^n\log\left(\frac{\cos (\alpha_i)}{\cos(\sqrt{\lambda}T+\alpha_i)}\right)&\text{if }\lambda>0\\
\sum_{i=1}^n\log\left(\frac{\alpha_i}{T+\alpha_i}\right) &\text{if }\lambda=0\\
\sum_{i=1}^n\log\left(\frac{\cosh (\alpha_i)}{\cosh(\sqrt{-\lambda}T+\alpha_i)}\right) &\text{if }\lambda<0\\
\end{cases}
\end{split}
\end{align}
where $\{\sigma_i\}_{i=1}^n$ are the eigenvalues of $-\nabla^2\log P_Tf(x)$.
\end{theorem}

\medskip

\subsection*{Acknowledgements} We are grateful to Dario Cordero-Erausquin,  Max Fathi, Nathael Gozlan, and Yury Polyanskiy for useful pointers to the literature and to  Georgios Moschidis for many helpful discussions. We also thank the anonymous referee for their helpful comments. This material is based upon work supported by the National
Science Foundation under Award Number 2002022.


\part{Euclidean and product spaces: scaling and tensorization}


\section{Logarithmic Sobolev inequalities in Euclidean spaces and Cram\'er--Rao bounds} \label{sec:lsi}

In this  section we discuss strengthenings of logarithmic Sobolev inequalities for measures on Euclidean spaces by means of scaling. In addition, we derive an application of these inequalities to Bayesian Cram\'er--Rao bounds.


\subsection{Warm-up: Gross' inequality}
The Euclidean reformulation \eqref{eq:lsi} of the logarithmic Sobolev inequality in Gauss space \cite{Gro75} asserts that if $f:\R^n\to\R_+$ is a probability density, then
\begin{equation} \label{eq:gross-2}
\int_{\R^n} f(x)\log f(x)\,\diff x - \HH(\gamma_n\|\lambda_n) \leq \frac{1}{2}\left( \int_{\R^n} \frac{|\nabla f(x)|^2}{f(x)} \,\diff x - n\right).
\end{equation}
Fix such a density $f$ and consider the reparametrized density $f_A:\R^n\to\R_+$ which is given by $f_A(x) = (\det A)\cdot f(Ax)$, where $A\in\msf{GL}_n$ is a positive definite matrix. Applying \eqref{eq:gross-2} for $f_A$ we get
\begin{equation}
\begin{split}
\int_{\R^n} f(x)\log f(x)\,\diff x & +\log\det A  - \HH(\gamma_n\|\lambda_n)=\int_{\R^n} f_A(x)\log f_A(x)\,\diff x - \HH(\gamma_n\|\lambda_n)
\\ & \leq \frac{1}{2}\left( \int_{\R^n} \frac{|\nabla f_A(x)|^2}{f_A(x)} \,\diff x - n\right) = \frac{1}{2}\left( \int_{\R^n} \frac{|A\cdot \nabla f(x)|^2}{f(x)} \,\diff x - n\right),
\end{split}
\end{equation}
which after rearranging becomes
\begin{equation} \label{eq:gross-3}
\begin{split}
\int_{\R^n} f(x)\log f(x)\,\diff x  - \HH(\gamma_n\|\lambda_n)& \leq \frac{1}{2}\left( \int_{\R^n} \frac{|A\cdot \nabla f(x)|^2}{f(x)} \,\diff x -\log\det A^2 - n\right)
\\ & =  \frac{1}{2}\left( \mathrm{tr} \big( A^2\cdot \III(\mu\|\lambda_n)\big) -\log\det A^2 - n\right).
\end{split}
\end{equation}
For the optimal choice of matrix $A=\III(\mu\|\lambda_n)^{-1/2}$, \eqref{eq:gross-3} readily becomes Dembo's inequality \eqref{eq:lsi**}. Observe that in this argument we made critical use of the change of variables formula for the Lebesgue measure, i.e., that $\lambda_n(AK) = (\det A)\cdot \lambda_n(K)$ for any Borel $K\subset\R^n$ and $A\in\msf{GL}_n$. While Lebesgue is the only measure on Euclidean space satisfying such an invariance property under all linear transformations, in the next section we shall observe that a weaker self-improvement can be deduced for measures which behave well under diagonal linear maps.


\subsection{Logarithmic Sobolev inequalities for homogeneous measures} \label{subsec:homogeneous}
Let $p_1,\ldots,p_n\geq0$. An absolutely continuous measure $\rho$ on $\R^n$ with density $w:\R^n\to\R_+$ is called $(p_1,\ldots,p_n)$-homogeneous if for every $t_1,\ldots,t_n>0$,
\begin{equation}
\forall \ x=(x_1,\ldots,x_n)\in\R^n, \qquad w(t_1x_1,\ldots,t_nx_n) = t_1^{p_1}\cdots t_n^{p_n} w(x_1,\ldots,x_n).
\end{equation}

\begin{theorem} \label{thm:homogeneous}
Fix $c_1,c_2>0$, $n\in\N$, $p_1,\ldots,p_n\geq0$ and let $\rho$ be a $(p_1,\ldots,p_n)$-homogeneous measure such that for any Borel probability measure $\mu$ on $\R^n$,
\begin{equation} \label{eq:homog-assumption}
\HH(\mu\|\rho)\leq c_1 \II(\mu\|\rho) + c_2.
\end{equation} 
Then, for any Borel probability measure $\mu$ on $\R^n$ with  positive differentiable density $f$, we have
\begin{equation}\label{eq:homog-conclusion}
\HH(\mu\|\rho)\leq  \frac{1}{2} \sum_{k=1}^n (1+p_k)  \log\left( \frac{2ec_1}{1+p_k} \int_{\R^n} \frac{(\partial_k f(y))^2}{f(y)} \,\diff\rho(y)\right) + c_2.
\end{equation}
\end{theorem}

The existence of homogeneous measures $\rho$ satisfying inequalities of the form \eqref{eq:homog-assumption}, as well as more general \emph{entropy-energy inequalities} follows, for instance, from \cite[Proposition~7.3.1]{BGL14}.

\begin{proof} [Proof of Theorem \ref{thm:homogeneous}]
Let $f=\frac{\diff\mu}{\diff \rho}$ be an arbitrary positive function with $\rho$-integral equal to 1 and fix $t_1,\ldots,t_n>0$. The measure $\mu_t$ with density $x\mapsto f_t(x)= t_1^{1+p_1}\cdots t_n^{1+p_n} f(t_1x_1,\ldots,t_nx_n)$ with respect to $\rho$ is a probability measure, as
\begin{equation}
\prod_{i=1}^n t_i^{1+p_i} \int_{\R^n} f(t_1x_1,\ldots,t_nx_n) w(x)\,\diff x = \prod_{i=1}^n t_i^{p_i}\int_{\R^n} f(y) w(t_1^{-1}y_1,\ldots,t_n^{-1}y_n)\,\diff y = 1,
\end{equation}
where we made the change of variables $(y_1,\ldots,y_n)=(t_1x_1,\ldots,t_nx_n)$. We have,
\begin{equation*}
\begin{split}
&\HH(\mu_t\|\rho)  =  \prod_{i=1}^n t_i^{1+p_i} \int_{\R^n} f(t_1x_1,\ldots,t_n x_n) \left\{\log f(t_1x_1,\ldots,t_nx_n) + \sum_{k=1}^n (1+p_k) \log t_k\right\} w(x)\,\diff x
\\ & =  \prod_{i=1}^n t_i^{p_i} \int_{\R^n} f(y) \left\{\log f(y)  + \sum_{k=1}^n (1+p_k) \log t_k\right\} w(t_1^{-1}y_1,\ldots,t_n^{-1}y_n) \,\diff y 
 = \HH(\mu\|\rho) + \sum_{k=1}^n (1+p_k) \log t_k.
\end{split}
\end{equation*}
Similarly, assuming in addition that $f$ is differentiable, for every $k\in\{1,\ldots,n\}$ we have $\partial_k f_t(x) = t_k \prod_i t_i^{1+p_i} \partial_kf(t_1x_1,\ldots,t_n x_n)$ and thus
\begin{equation}
\begin{split}
\II(\mu_t\|\rho) & = \sum_{k=1}^n t_k^2 \cdot \prod_{i=1}^n t_i^{1+p_i} \int_{\R^n} \frac{(\partial_k f(t_1x_1,\ldots,t_nx_n))^2}{f(t_1x_1,\ldots,t_nx_n)} w(x)\,\diff x
\\ & = \sum_{k=1}^n t_k^2 \cdot\prod_{i=1}^n t_i^{p_i} \int_{\R^n} \frac{(\partial_k f(y))^2}{f(y)} w(t_1^{-1}y_1,\ldots,t_n^{-1}y_n) \,\diff y = \sum_{k=1}^n t_k^2 \int_{\R^n} \frac{(\partial_k f(y))^2}{f(y)} \,\diff\rho(y).
\end{split}
\end{equation}
Therefore, applying \eqref{eq:homog-assumption} for $\mu_t$ and reorganizing the terms, we deduce that
\begin{equation*}
\HH(\mu\|\rho) \leq \inf_{t_1,\ldots,t_n\geq0} \left\{ c_1 \sum_{k=1}^n t_k^2 \int_{\R^n} \frac{(\partial_k f(y))^2}{f(y)} \,\diff\rho(y) - \sum_{k=1}^n (1+p_k) \log t_k + c_2 \right\}.
\end{equation*}
It is now elementary to check that the above infimum is attained when
\begin{equation}
\forall \ k\in\{1,\ldots,n\},\qquad t_k^2 = \frac{1+p_k}{2c_1} \Big(  \int_{\R^n} \frac{(\partial_k f(y))^2}{f(y)} \,\diff\rho(y) \Big)^{-1}
\end{equation}
and plugging this choice of parameters completes the proof.
\end{proof}

Specifically for Lebesgue measure, Theorem \ref{thm:homogeneous} implies that if $\mu$ has differentiable density $f$, 
\begin{equation} \label{eq:lsi-diag}
\HH(\mu\|\lambda_n)-\HH(\gamma_n\|\lambda_n) \leq \frac{1}{2} \sum_{k=1}^n \log\left(  \int_{\R^n} \frac{(\partial_k f(y))^2}{f(y)} \,\diff y \right),
\end{equation}
which is weaker than Dembo's inequality \eqref{eq:lsi**} in view of the elementary estimate $\det C\leq \prod_s C_{ss}$ which holds for all positive semidefinite matrices $C$. On the other hand, \eqref{eq:lsi-diag} combined with Jensen's inequality implies \eqref{eq:lsi_dim}. We refer to \cite{BCL95,BL06,BBL12,BBL17} for further dimensional logarithmic Sobolev inequalities and applications to Li--Yau-type estimates \cite{LY86}, hypercontractivity \cite{Nel66, Bon70, Nel73, Bec75, Bak94} and heat kernel estimates \cite{Bak94,BCL95}.
 

\subsection{A Bayesian Cram\'er--Rao bound} 
\label{subsec:CR_bound}
In \cite{ALPC19}, Aras, Lee, Pananjady and Courtade observed that logarithmic Sobolev inequalities formally imply Bayesian Cram\'er--Rao bounds, thus extending some results of Efroimovich \cite{Efr80} for Gaussian measures. In this section, we investigate similar applications of intrinsic dimensional log-Sobolev inequalities in the spirit of \eqref{eq:homog-conclusion} and \eqref{eq:lsi**}.

Following \cite{ALPC19}, we work in the setting of parametric statistics. Let $\{\mu_\theta\}_{\theta\in\R^n}$ be a family of probability measures on a measurable space $(\Omega,\ms{F})$. Assume moreover that there exists a dominating $\sigma$-finite measure $\lambda$ on $\Omega$ such that $\mu_\theta$ has a positive density with respect to $\lambda$,
\begin{equation}
\forall \ \theta\in\R^n, \qquad \diff\mu_\theta(x) = f(x;\theta)\,\diff\lambda(x).
\end{equation}
We shall assume throughout that each function $\theta\mapsto f(x;\theta)$ is smooth and that
\begin{equation} \label{eq:regu}
\int_\Omega \nabla_\theta f(x;\theta) \,\diff\lambda(x) = 0
\end{equation}
for almost every $\theta\in\R^n$. The Fisher information of the parametric family $\{\mu_\theta\}_{\theta\in\R^n}$ is
\begin{equation}
\forall \ \theta\in\R^n,\qquad \ms{J}(\theta) \eqdef \int_\Omega \frac{|\nabla_\theta f(x;\theta)|^2}{f(x;\theta)}\,\diff\lambda(x).
\end{equation}
Finally, if $\pi$ is a probability measure on $\R^n$, we denote the mutual information of $\pi$ with the family $\{\mu_\theta\}_{\theta\in\R^n}$ by
\begin{equation}
I\big(\pi;\{\mu_\theta\}\big) \eqdef \int_{\R^n} \int_\Omega f(x;\theta) \log \left(\frac{f(x;\theta)}{\int_{\R^n} f(x,\phi)\,\diff\pi(\phi)}\right) \,\diff\lambda(x) \,\diff\pi(\theta).
\end{equation}
The main result of \cite[Theorem~1]{ALPC19} specified to the standard Gaussian measure $\gamma_n$ asserts that for every absolutely continuous probability measure $\pi$ on $\R^n$,
\begin{equation} \label{eq:alpc}
I\big(\pi;\{\mu_\theta\}\big) + \HH(\pi\|\gamma_n) \leq \frac{1}{2}\Big( \II(\pi\|\gamma_n) + \int_{\R^n} \ms{J}(\theta)\,\diff\pi(\theta) \Big).
\end{equation}
Inequality \eqref{eq:alpc} implies the Gaussian logarithmic Sobolev inequality \eqref{eq:lsi-gross} since choosing $\mu_\theta=\lambda$ independently of $\theta$, the terms $I(\pi;\{\mu_\theta\})$ and $\ms{J}(\theta)$ both vanish. We present inequalities in the spirit of \eqref{eq:alpc} for homogeneous measures satisfying a log-Sobolev inequality of the form \eqref{eq:homog-assumption}.

\begin{theorem} \label{thm:cramer-rao}
Fix $c_1,c_2>0$, $n\in\N$, $p_1,\ldots,p_n\geq0$ and let $\rho$ be a $(p_1,\ldots,p_n)$-homogeneous measure  such that for any probability measure $\mu$ on $\R^n$,
\begin{equation} \label{eq:cr-assumption}
\HH(\mu\|\rho)\leq c_1 \II(\mu\|\rho) + c_2.
\end{equation}
Then, for every parametric family $\{\mu_\theta\}_{\theta\in\R^n}$ and every absolutely continuous probability measure $\pi$ on $\R^n$ whose density with respect to $\rho$ is $h:\R^n\to\R_+$, we have
\begin{equation}
\begin{split}
&I\big(\pi;\{\mu_\theta\}\big)+ \HH(\pi\|\rho)
\\& \leq \frac{1}{2}\sum_{k=1}^n (1+p_k) \log\left( \frac{2ec_1}{1+p_k}\Big(\int_{\R^n} \frac{(\partial_{k} h(\theta))^2}{h(\theta)}\,\diff\rho(\theta) + \int_{\R^n} \int_\Omega \frac{(\partial_{\theta_k} f(x;\theta))^2}{f(x;\theta)} \,\diff\lambda(x)\,\diff\pi(\theta)\Big) \right) + c_2.
\end{split}
\end{equation}
\end{theorem}

Observe that the terms inside the logarithm on the right-hand side are the $k$-th component of the Fisher informations $\II(\pi|\rho)$ and $\ms{J}(\theta)$ respectively, in analogy with Theorem \ref{thm:homogeneous}.

\begin{proof} [Proof of Theorem \ref{thm:cramer-rao}]
Consider the function $f:\Omega\to\R_+$ given by
\begin{equation} \label{eq:cramer-f}
\forall \ x\in\Omega,\qquad f(x)\eqdef \int_{\R^n} f(x;\theta) \,\diff\pi(\theta)
\end{equation}
and observe that
\begin{equation}
\int_\Omega f(x)\,\diff\lambda(x) = \int_\Omega \int_{\R^n} f(x;\theta) \,\diff\pi(\theta)\,\diff\lambda(x) = \int_{\R^n} \int_\Omega \, \diff\mu_\theta(x) \diff\pi(\theta) = 1.
\end{equation}
Moreover, for $x\in\Omega$, consider the function $h_x:\R^n\to\R_+$ given by
\begin{equation}
\forall \ \theta\in\R^n, \qquad h_x(\theta) \eqdef \frac{h(\theta)f(x;\theta)}{f(x)}
\end{equation}
and notice that the measure $\nu_x$ on $\R^n$ with $\diff\nu_x(\theta) = h_x(\theta)\diff \rho(\theta)$ is a probability measure since
\begin{equation}
\nu_x(\R^n) = \int_{\R^n} \frac{h(\theta) f(x;\theta)}{f(x)} \, \diff\rho(\theta) = \int_{\R^n} \frac{f(x;\theta)}{f(x)} \,\diff\pi(\theta)  \stackrel{\eqref{eq:cramer-f}}{=}1.
\end{equation}
By Theorem \ref{thm:homogeneous} and the assumption on $\rho$, for every $x\in\Omega$ we have
\begin{equation}
\HH(\nu_x\|\rho) \leq \frac{1}{2} \sum_{k=1}^n (1+p_k)\log\Big( \frac{2ec_1}{1+p_k} \int_{\R^n} \frac{(\partial_k h_x(\theta))^2}{h_x(\theta)}\,\diff\rho(\theta)\Big) + c_2.
\end{equation}
Integrating this inequality with respect to the probability measure $f(x)\diff\lambda(x)$, we get
\begin{equation} \label{eq:rao-long}
\begin{split}
\int_\Omega \HH(\nu_x\|\rho)f(x) \,\diff\lambda(x) & \leq \frac{1}{2} \sum_{k=1}^n (1+p_k) \int_\Omega\log\Big( \frac{2ec_1}{1+p_k} \int_{\R^n} \frac{(\partial_k h_x(\theta))^2}{h_x(\theta)}\,\diff\rho(\theta)\Big)f(x)\,\diff\lambda(x) + c_2
\\ & \leq \frac{1}{2} \sum_{k=1}^n (1+p_k) \log\Big( \frac{2ec_1}{1+p_k} \int_\Omega \int_{\R^n} \frac{(\partial_k h_x(\theta))^2}{h_x(\theta)}f(x) \,\diff\rho(\theta)\,\diff\lambda(x)\Big) + c_2,
\end{split}
\end{equation}
where the last line follows from Jensen's inequality. Moreover, by definition we have
\begin{equation}
\begin{split}
\int_\Omega \HH(\nu_x\|\rho)f(x) \,\diff\lambda(x) & = \int_\Omega \int_{\R^n} h(\theta) f(x;\theta) \log\Big( \frac{h(\theta)f(x;\theta)}{f(x)}\Big) \, \diff\rho(\theta)\,\diff\lambda(x) 
\\ & = \int_{\R^n} h(\theta) \log h(\theta)\,\diff\rho(\theta) + \int_{\R^n} \int_\Omega f(x;\theta) \log \big(\frac{f(x;\theta)}{f(x)} \big)\, \diff\lambda(x)\,\diff\pi(\theta)
\\ & = \HH(\pi\|\rho) + I\big(\pi;\{\mu_\theta\}\big).
\end{split}
\end{equation}
Similarly, computing the integral on the right-hand side of \eqref{eq:rao-long}, gives
\begin{equation*}
\begin{split}
 \int_\Omega \int_{\R^n} &\frac{(\partial_k h_x(\theta))^2}{h_x(\theta)}f(x) \,\diff\rho(\theta)\,\diff\lambda(x) =  \int_\Omega \int_{\R^n} \frac{\big(f(x;\theta)\partial_k h(\theta) + h(\theta) \partial_{\theta_k} f(x;\theta)\big)^2}{h(\theta)f(x;\theta)} \,\diff\rho(\theta)\,\diff\lambda(x)
 \\ & = \int_{\R^n}\frac{(\partial_k h(\theta))^2}{h(\theta)} \int_\Omega f(x;\theta) \, \diff\lambda(x)\,\diff\rho(\theta) + 2 \int_{\R^n} \partial_k h(\theta) \int_\Omega \partial_{\theta_k} f(x;\theta) \,\diff\lambda(x) \,\diff\rho(\theta)
 \\ & \qquad\qquad\qquad\qquad\qquad\qquad\qquad\qquad\qquad + \int_{\R^n} h(\theta) \int_\Omega \frac{(\partial_{\theta_k}f(x;\theta))^2}{f(x;\theta)}\,\diff\lambda(x)\,\diff\rho(\theta)
 \\ & \stackrel{\eqref{eq:regu}}{=} \int_{\R^n} \frac{(\partial_{k} h(\theta))^2}{h(\theta)}\,\diff\rho(\theta) + \int_{\R^n} \int_\Omega \frac{(\partial_{\theta_k} f(x;\theta))^2}{f(x;\theta)} \,\diff\lambda(x)\,\diff\pi(\theta).
\end{split}
\end{equation*}
Combining everything, we deduce the desired inequality.
\end{proof}

\begin{remark}
In the case of the Gaussian measure $\rho=\gamma_n$, we have at our disposal the intrinsic dimensional logarithmic Sobolev inequality \eqref{eq:lsi**}. Repeating the same proof mutatis mutandis while replacing \eqref{eq:homog-conclusion} with \eqref{eq:lsi**}, we conclude that for any probability measure $\pi$ on $\R^n$ whose density with respect to $\gamma_n$ is $h:\R^n\to\R_+$, and for every parametric family $\{\mu_\theta\}_{\theta\in\R^n}$, we have
\begin{equation*} \label{eq:efroi}
I\big(\pi;\{\mu_\theta\}\big) + \HH(\pi\|\gamma_n) \leq \frac{\mr{tr}\MM_{2,\pi}-n}{2} + \frac{1}{2}\log\det\Big( 2\msf{Id}_n + \III(\pi\|\gamma_n) + \!\!\int_{\R^n}\!\!\! \int_\Omega \frac{(\nabla_\theta f(x;\theta))^{\otimes 2}}{f(x;\theta)}\diff\lambda(x) \diff\pi(\theta) - \MM_{2,\pi} \Big),
\end{equation*}
where $\MM_{2,\pi} = \int \theta^{\otimes 2}\diff\pi(\theta)$. This recovers a result of Efroimovich \cite[Theorem~5]{Efr80}. Combining the inequalities $\log\det C \leq n\log\tfrac{\mr{tr}C}{n}$ and $\log y \leq y-1$, which hold for all $y>0$ and all $n\times n$ positive definite matrices $C$, we see that Efroimovich's inequality is a strengthening of \eqref{eq:alpc}.
\end{remark}


\section{Gagliardo--Nirenberg--Sobolev inequalities} \label{sec:gns}
In this section we shall prove Theorem \ref{thm:gns_intro}:

\begin{theorem} \label{thm:gns}
Let $p,q,r\in[1,\infty)$, $\theta\in[0,1]$ and $C^{p,q,r,r}>0$ be such that 
\begin{equation} \label{eq:gns-gen}
\|u\|_{L_p(\R^n)} \leq C^{p,q,r,r}\|u\|_{L_q(\R^n)}^\theta \|\nabla u\|_{L_r(\R^n;\ell_r^n)}^{1-\theta}
\end{equation}
is satisfied for all functions $u\in C_0^\infty(\R^n)$ under the constraint
\begin{equation} \label{eq:gn}
\frac{1}{p} = \frac{\theta}{q}+ \Big(\frac{1}{r}-\frac{1}{n}\Big) (1-\theta).
\end{equation}  
Then, for every $u\in C_0^\infty(\R^n)$, we have
\begin{equation} \label{eq:gns*}
\|u\|_{L_p(\R^n)} \leq C^{p,q,r,r} n^{\frac{1-\theta}{r}} \|u\|_{L_q(\R^n)}^\theta \Big( \prod_{j=1}^n \|\partial_ju\|_{L_r(\R^n)}\Big)^{\frac{1-\theta}{n}}.
\end{equation}
\end{theorem}

\begin{proof}
Fix $t=(t_1,\ldots,t_n)\in\R_+^n$ and consider the function $u_t\in C_0^\infty(\R^n)$ given by 
\begin{equation}
\forall \ x=(x_1,\ldots,x_n)\in\R^n, \qquad u_t(x)\eqdef u(t_1x_1,\ldots,t_nx_n).
\end{equation}
Then, for $s\geq1$ we have 
\begin{equation}
\|u_t\|_{L_s(\R^n)} = \Big( \int_{\R^n} |u(t_1x_1,\ldots,t_nx_n)|^s\,\diff x \Big)^{1/s} = \prod_{j=1}^n t_j^{-1/s} \|u\|_{L_s(\R^n)}
\end{equation}
and
\begin{equation}
\|\nabla u_t\|_{L_r(\R^n;\ell_r^n)}^r = \sum_{i=1}^n \int_{\R^n} \|\partial_i u_t\|_{L_r(\R^n)}^r = \sum_{i=1}^n t_i^r \prod_{j=1}^n t_j^{-1} \|\partial_i u\|_{L_r(\R^n)}^r.
\end{equation}
Therefore, applying \eqref{eq:gns-gen} to $u_t$ and rearranging, we deduce that
\begin{equation}
\begin{split}
\|u\|_{L_p(\R^n)} & \leq C^{p,q,r,r} \Big( \prod_{j=1}^nt_j \Big)^{\frac{1}{p}-\frac{\theta}{q}-\frac{1-\theta}{r}} \|u\|_{L_q(\R^n)}^\theta \Big( \sum_{i=1}^n t_i^r \|\partial_i u\|_{L_r(\R^n)}^r\Big)^{\frac{1-\theta}{r}}
\\ & \stackrel{\eqref{eq:gn}}{=} C^{p,q,r,r} \Big( \prod_{j=1}^n t_j\Big)^{-\frac{1-\theta}{n}} \|u\|_{L_q(\R^n)}^\theta \Big( \sum_{i=1}^n t_i^r \|\partial_i u\|_{L_r(\R^n)}^r\Big)^{\frac{1-\theta}{r}}
\end{split}
\end{equation}
for every $t_1,\ldots,t_n>0$. Choosing
\begin{equation}
t_i = \|\partial_iu\|_{L_r(\R^n)}^{-1}
\end{equation} 
gives the desired inequality \eqref{eq:gns*}.
\end{proof}


\section{Beckner inequalities} \label{sec:beckner}

In this section we shall prove Theorem \ref{thm:beckner}:

\begin{theorem} \label{thm:beckner_sec}
Fix $n\in\N$ and let $u\in C_0^\infty(\R^n)$ be such that
\begin{equation}
\forall \ i,j\in\{1,\ldots,n\},\qquad \int_{\R^n} x_i x_j u(x)^2 \,\diff\gamma_n(x) = \delta_{ij} \|u\|_{L_2(\gamma_n)}^2,
\end{equation}
where $\delta_{ij}$ is the Kronecker delta. Then, we have
\begin{equation} \label{eq:dt*_sec}
\forall \ p\in[1,2),\qquad \frac{\|u\|_{L_2(\gamma_n)}^2-\|u\|_{L_p(\gamma_n)}^2}{\|u\|_{L_2(\gamma_n)}^2} \leq 1-\det\left(\frac{4}{\|u\|_{L_2(\gamma_n)}^2} \int_{\R^n} (\nabla u)^{\otimes 2} \,\diff\gamma_n + \msf{Id}_n\right)^{-\frac{2-p}{2p}}.
\end{equation}
\end{theorem}

For the proof of Theorem \ref{thm:beckner_sec} we shall use the intrinsic dimensional logarithmic Sobolev inequality \eqref{eq:lsi**} which takes the following simple form for appropriately normalized functions in Gauss space.

\begin{lemma} \label{lem:beckner}
Let $u\in C_0^\infty(\R^n)$ be such that $\|u\|_{L_2(\gamma_n)}=1$ and 
\begin{equation} \label{eq:lem-constr}
\forall \ i,j\in\{1,\ldots,n\},\qquad \int_{\R^n} x_i x_j u(x)^2 \,\diff\gamma_n(x) = \delta_{ij}.
\end{equation}
Then, we have
\begin{equation} \label{eq:lem-beckner}
\mathrm{Ent}_{\gamma_n}[u^2] \leq \frac{1}{2} \log\det \left( 4 \int_{\R^n} (\nabla u)^{\otimes 2} \,\diff \gamma_n + \msf{Id}_n \right).
\end{equation}
\end{lemma}

\begin{proof}
Let $u\in C_0^\infty(\R^n)$ satisfy the assumptions of the lemma and define $f(x) = u(x)^2 \tfrac{\exp(-|x|^2/2)}{(2\pi)^{n/2}}$ which is the density of a probability measure $\mu$ on $\R^n$. Then, we have
\begin{equation}
\HH(\mu\|\lambda_n)-\HH(\gamma_n\|\lambda_n) = \mr{Ent}_{\gamma_n}[u^2] - \frac{1}{2} \int_{\R^n} |x|^2 u(x)^2 \,\diff\gamma_n(x) - \frac{n}{2}\log2\pi + \frac{n}{2}\log2\pi e \stackrel{\eqref{eq:lem-constr}}{=} \mr{Ent}_{\gamma_n}[u^2].
\end{equation}
On the other hand, for $k\in\{1,\ldots,n\}$, we compute
\begin{equation}
\partial_kf(x) =\big( 2u(x)\partial_ku(x) - x_k u(x)^2\big) \frac{e^{-|x|^2/2}}{(2\pi)^{n/2}}
\end{equation}
and thus for $i,j\in\{1,\ldots,n\}$, we get
\begin{equation} \label{eq:form-for-I}
\III(\mu\|\lambda_n)_{ij} = 4 \int_{\R^n} \partial_iu \partial_j u\,\diff\gamma_n - 2\int_{\R^n} \big( x_j \partial_iu(x)+x_i\partial_ju(x)\big) u(x)\,\diff\gamma_n(x) + \int_{\R^n} x_ix_j u(x)^2\,\diff\gamma_n(x).
\end{equation}
For $i\neq j$, integration by parts gives
\begin{equation}
2\int_{\R^n}  x_j \partial_iu(x) u(x)\,\diff\gamma_n(x) = \int_{\R^n} \partial_i \big( x_j u(x)^2\big)\,\diff\gamma_n(x) = \int_{\R^n} x_ix_j u(x)^2\,\diff\gamma_n(x) \stackrel{\eqref{eq:lem-constr}}{=} 0,
\end{equation}
whereas for $i=j$, again by integration by parts,
\begin{equation*}
2\int_{\R^n}  x_i \partial_iu(x) u(x)\,\diff\gamma_n(x) = \int_{\R^n} \partial_i \big( x_i u(x)^2\big)\,\diff\gamma_n(x) - \int_{\R^n} u(x)^2\,\diff\gamma_n(x) = \int_{\R^n} x_i^2 u(x)^2\,\diff\gamma_n(x) - 1 \stackrel{\eqref{eq:lem-constr}}{=} 0.
\end{equation*}
Plugging the above in \eqref{eq:form-for-I} and using \eqref{eq:lem-constr} again for the last term, we deduce that
\begin{equation}
\III(\mu\|\lambda_n) = 4 \int_{\R^n} (\nabla u)^{\otimes 2} \,\diff\gamma_n + \msf{Id}_n
\end{equation}
and the conclusion of the lemma follows from \eqref{eq:lsi**}.
\end{proof}

Equipped with Lemma \ref{lem:beckner}, we proceed to the proof of Theorem \ref{thm:beckner}.

\begin{proof} [Proof of Theorem \ref{thm:beckner_sec}]
Assume, without loss of generality, that $\|u\|_{L_2(\gamma_n)}=1$. Combining a lemma of Dolbeault and Toscani \cite[Lemma~5]{DT16} (see also \cite{LO00}) with Lemma \ref{lem:beckner}, we get that
\begin{equation}
\frac{1}{\|u\|_{L_p(\gamma_n)}^2} = \frac{\|u\|_{L_2(\gamma_n)}^2}{\|u\|_{L_p(\gamma_n)}^2} \leq \exp \left(\frac{2-p}{p} \mr{Ent}_{\gamma_n}[u^2]\right) \stackrel{\eqref{eq:lem-beckner}}{\leq} \det \left( 4 \int_{\R^n} (\nabla u)^{\otimes 2} \,\diff \gamma_n + \msf{Id}_n \right)^{\frac{2-p}{2p}}.
\end{equation}
Therefore,
\begin{equation}
1-\|u\|_{L_p(\gamma_n)}^2 \leq 1-\det \left( 4 \int_{\R^n} (\nabla u)^{\otimes 2} \,\diff \gamma_n + \msf{Id}_n \right)^{-\frac{2-p}{2p}},
\end{equation}
which is the desired estimate under the normalization $\|u\|_{L_2(\gamma_n)}=1$.
\end{proof}


\section{$q$-logarithmic Sobolev inequalities} \label{sec:qlsi}

Following Bobkov and Zegarlinski \cite{BZ05} (see also \cite{BR08}) we say that a probability measure $\mu$ on the real line satisfies the $q$-logarithmic Sobolev inequality with constant $C>0$ if for any $f\in C_0^\infty(\R)$ we have
\begin{equation} \label{eq:qlsi}
\mathrm{Ent}_\mu [ |f|^q ] \leq C \int_\R  |f'(x)|^q\,\diff\mu(x).
\end{equation}
Standard tensorization principles show that if \eqref{eq:qlsi} holds, then for any $f\in C_0^\infty(\R^n)$,
\begin{equation} \label{eq:qlsi'}
\mathrm{Ent}_{\mu^{n}} [|f|^q] \leq C \sum_{i=1}^n\int_{\R^n} |\partial_i f(x)|^q\, \diff\mu^n(x),
\end{equation}
where $\mu^n = \mu\otimes\mu\otimes\cdots\otimes\mu$ is the product measure of i.i.d. coordinates distributed like $\mu$.  In particular, it has been established in \cite[Corollary~5.6]{BZ05} (see also \cite[Section~5]{BL00}) that the measure $\mu_p$ with density $\tfrac{1}{Z_p}e^{-|x|^p}$, where $p>2$, satisfies the $q$-logarithmic Sobolev inequality for $q=\tfrac{p}{p-1}$ with some constant $C_q>1$.  In order to investigate scale-invariant refinements of \eqref{eq:qlsi'} for this family of measures in the spirit of \eqref{eq:lsi-diag}, we first need to formulate them as Euclidean inequalities.

\begin{theorem}
For any $q\in(1,2)$, there exists a constant $\tilde{C}_q>0$  such that for any $n\in\N$ and any probability measure $\mu$ on $\R^n$ with positive differentiable density $g$,
\begin{equation} \label{eq:qlsi''}
\msf{H}(\mu \| \lambda_n) \leq \tilde{C}_q \sum_{i=1}^n \int_{\R^n} \left(\Big| \frac{\partial_i g(x)}{g(x)} \Big|^q + |x_i|^\frac{q}{q-1}\right) \,\diff\mu(x).
\end{equation}
\end{theorem}

\begin{proof}
For $p=\frac{q}{q-1}>2$ consider the probability measure $\diff\mu_p(x) = \tfrac{e^{-|x|^p}}{Z_p}$ on $\R$, where the normalizing constant is $Z_p = 2\Gamma(1+1/p)>2$.  Let $\mu$ be a probability measure on $\R^n$ with differentiable density $g:\R^n\to\R_+$ and consider the function $f:\R^n\to\R_+$ given by
\begin{equation}
\forall \ x\in\R^n,\qquad f(x)  = Z_p^{n/q} g(x)^{1/q} e^{\|x\|_p^p/q},
\end{equation}
which satisfies $\int_{\R^n} f(x)^q \,\diff\mu_p^n(x) = 1$. Therefore,  the $q$-logarithmic Sobolev inequality for $\mu_p^n$ applied to the function $f$ implies that
\begin{equation} \label{eq:applied-q-lsi}
\frac{1}{Z_p^n} \int_{\R^n} f(x)^q e^{-\|x\|_p^p} \log f(x)^q \,\diff x = \mathrm{Ent}_{\mu_p^n} [f^q] \leq \frac{C_q}{Z_p^n} \sum_{i=1}^n \int_{\R^n} |\partial_i f(x)|^q e^{-\|x\|_p^p} \, \diff x.
\end{equation}
Observe that
\begin{equation}
\begin{split}
\frac{1}{Z_p^n} \int_{\R^n}  f(x)^q e^{-\|x\|_p^p} \log f(x)^q \,\diff x & = \int_{\R^n} g(x) \log\big(Z_p^n g(x) e^{\|x\|_p^p}\big)\,\diff x
\\ & = \msf{H}(\mu\|\lambda_n) + \int_{\R^n} \|x\|_p^p\,\diff\mu(x) + \log Z_p^n,
\end{split}
\end{equation}
and for $i\in\{1,\ldots,n\}$,
\begin{equation}
\frac{1}{Z_p^n} \int_{\R^n} |\partial_i f(x)|^q e^{-\|x\|_p^p} \,\diff x = \frac{1}{q^q} \int_{\R^n} \Big| \frac{\partial_i g(x)}{g(x)}+ p \mathrm{sign}(x_i)|x_i|^{p-1}\Big|^q \,\diff\mu(x).
\end{equation}
Therefore, rearranging \eqref{eq:applied-q-lsi} we deduce that
\begin{equation}
\begin{split}
\msf{H}(\mu\|\lambda_n) & \leq \frac{C_q}{q^q} \sum_{i=1}^n\int_{\R^n} \Big| \frac{\partial_i g(x)}{g(x)}+ p \mathrm{sign}(x_i)|x_i|^{p-1}\Big|^q \,\diff\mu(x) - \sum_{i=1}^n \int_{\R^n} |x_i|^p\,\diff\mu(x) - \log Z_p^n \\ & \leq \tilde{C}_q \sum_{i=1}^n \int_{\R^n} \left(\Big| \frac{\partial_i g(x)}{g(x)} \Big|^q + |x_i|^p\right) \,\diff\mu(x) 
\end{split}
\end{equation}
for some different constant $\tilde{C}_q>0$ and the proof is complete.
\end{proof}

This Euclidean weakening of the $q$-logarithmic Sobolev inequality \eqref{eq:qlsi'} for $\mu_p^n$ makes it amenable to refinements via scaling.

\begin{theorem}
For any $q\in(1,2)$ and $p=\tfrac{q}{q-1}$, there exists a constant $\tilde{C}_q>0$ such that for any $n\in\N$ and any probability measure $\mu$ on $\R^n$,
\begin{equation} \label{eq:qlsi'''}
\msf{H}(\mu\|\lambda_n) \leq \sum_{i=1}^n \ \inf_{t_i>0} \Big\{ \tilde{C}_q t_i^q \int_{\R^n} \Big| \frac{\partial_ig(x)}{g(x)}\Big|^q \,\diff\mu(x) + \frac{\tilde{C}_q}{t_i^p} \int_{\R^n} |x_i|^p\,\diff\mu(x) - \log t_i \Big\}.
\end{equation}
\end{theorem}

\begin{proof}
Fix $t_1,\ldots,t_n>0$ and consider the probability measure $\mu_t$ whose density is given by $x\mapsto g_t(x)\eqdef t_1\cdots t_n g(t_1x_1,\ldots,t_nx_n)$. Then, we have
\begin{equation}
\msf{H}(\mu_t\|\lambda_n) = \msf{H}(\mu\|\lambda_n)+ \sum_{i=1}^n \log t_i
\end{equation}
and for every $i\in\{1,\ldots,n\}$,
\begin{equation}
\int_{\R^n} \left(\Big| \frac{\partial_i g_t(x)}{g_t(x)}\Big|^q + |x_i|^p \right)\,\diff\mu_t(x) = \int_{\R^n} \left(t_i^q \Big| \frac{\partial_i g(x)}{g(x)}\Big|^q + \frac{|x_i|^p}{t_i^p}\right) \,\diff\mu(x).
\end{equation}
Therefore, applying \eqref{eq:qlsi''} to $\mu_t$ and rearranging, we deduce that
\begin{equation}
\msf{H}(\mu\|\lambda_n) \leq \sum_{i=1}^n  \Big\{ \tilde{C}_q t_i^q \int_{\R^n} \Big| \frac{\partial_ig(x)}{g(x)}\Big|^q \,\diff\mu(x) + \frac{\tilde{C}_q}{t_i^p} \int_{\R^n} |x_i|^p\,\diff\mu(x) - \log t_i \Big\}
\end{equation}
and taking an infimum over $t_1,\ldots,t_n>0$ completes the proof.
\end{proof}


\section{Beyond linear rescalings}

The simple idea of the previous sections can be summarized as follows.  Let 
\begin{equation} \label{eq:nonlinear-proto}
\ms{K}(f)\leq \ms{L}(f)
\end{equation} 
be a functional inequality valid for regular enough functions $f$ on $\R^n$ and fix a subgroup of symmetries $\msf{G}\subseteq\msf{GL}_n$. For a fixed $f:\R^n\to\R$ for which inequality \eqref{eq:nonlinear-proto} is valid and $A\in\msf{G}$, consider the function $f_A:\R^n\to\R$ given by $f_A(x) = f(Ax)$. \emph{If} \eqref{eq:nonlinear-proto} applied to $f_A$ can be rearranged to an upper bound for $\ms{K}(f)$ of the form
\begin{equation}
\ms{K}(f) \leq \widetilde{L}(f,A),
\end{equation}
then taking an infimum over $A\in\msf{G}$ yields a stronger inequality as \eqref{eq:nonlinear-proto} just amounts to the choice $A=\msf{Id}_n$. Observe that enhancing inequalities in this way, always produces a larger family of extremals. For instance, \eqref{eq:lsi} becomes an equality only when $\mu$ is a translate of $\gamma_n$, \eqref{eq:lsi_dim} becomes an equality when $\mu$ is a Gaussian measure with covariance matrix of the form $\sigma\msf{Id}_n$, where $\sigma>0$, and \eqref{eq:lsi**} becomes an equality for any Gaussian measure on $\R^n$.

In this section, we will discuss the possibility of refining functional inequalities by using changes of variables via \emph{nonlinear} maps and we shall illustrate this in the case of the logarithmic Sobolev inequality \eqref{eq:lsi}. Let $T:\R^n\to\R^n$ be a smooth diffeomorphism and for a measure $\mu$ on $\R^n$ with a differentiable density $f:\R^n\to\R_+$ consider the measure $\mu_T$ whose density is given by $f_T(x) = (f\circ T)(x) |\det \DD T(x)|$, where $x\in\R^n$ and $\DD T\in\msf{M}_n(\R)$ is the differential of $T$. We  need the following computations for the relative entropy and \mbox{Fisher information of $\mu_T$.}

\begin{lemma} \label{lem:mT}
In the setting above,
\begin{equation} \label{eq:HmT}
\HH(\mu_T\|\lambda_n) = \HH(\mu\|\lambda_n) + \int_{\R^n} \log|\det \DD T(T^{-1}(x))| \,\diff\mu(x)
\end{equation}
and
\begin{equation} \label{eq:ImT}
\II(\mu_T\|\lambda_n) = \int_{\R^n} \left| \frac{\msf{D}T(T^{-1}(x)) \cdot \nabla f(x)}{f(x)} +\nabla \log|\det \msf{D}T(T^{-1}(x)) |\right|^2 \,\diff\mu(x).
\end{equation}
\end{lemma}

The proof is a straightforward computation using a change of variables and is thus omitted. These formulas along with the fact that any absolutely continuous measure can be transported to $\gamma_n$ give rise to the following variational formula for relative entropy on $\R^n$.

\begin{theorem} \label{thm:mT}
Let $\mu$ be an absolutely continuous measure on $\R^n$. Then,
\begin{equation} \label{eq:entropy-ident}
\HH(\mu\|\lambda_n)-\HH(\gamma_n\|\lambda_n) = \min_{T\in\mr{Diff}(\R^n)} \psi(T)
\end{equation}
with equality if $T$ is a transport map from $\mu$ to $\gamma_n$, where
\begin{equation*}
\psi(T) \eqdef  \frac{1}{2} \int_{\R^n} \left| \frac{\msf{D}T(T^{-1}(x)) \cdot \nabla f(x)}{f(x)} +\nabla \log|\det \msf{D}T(T^{-1}(x)) |\right|^2 \,\diff\mu(x) - \int_{\R^n} \log|\det \DD T(T^{-1}(x))| \,\diff\mu(x).
\end{equation*}
\end{theorem}

\begin{proof}
Applying the logarithmic Sobolev inequality \eqref{eq:lsi} to $\mu_T$ and using Lemma \ref{lem:mT}, we get
\begin{equation}
\begin{split}
 \HH(\mu\|\lambda_n) + \int_{\R^n}& \log|\det \DD T(T^{-1}(x))| \,\diff\mu(x) - \HH(\gamma_n\|\lambda_n) 
 \\ & \leq \frac{1}{2} \int_{\R^n} \left| \frac{\msf{D}T(T^{-1}(x)) \cdot \nabla f(x)}{f(x)} +\nabla \log|\det \msf{D}T(T^{-1}(x)) |\right|^2 \,\diff\mu(x),
 \end{split}
\end{equation}
with equality only if $\mu_T=\gamma_n$. The existence of a map $T$ transporting $\mu$ to $\gamma_n$ is a classical fact in optimal transport going back to at least \cite{Ros52,Kno57} (see also \cite{Vil09}).
\end{proof}

We are not aware of a proof of \eqref{eq:entropy-ident} which does not rely on the logarithmic Sobolev inequality \eqref{eq:lsi}. It remains very interesting to understand whether \eqref{eq:entropy-ident} can lead to stability estimates for \eqref{eq:lsi}, or even \eqref{eq:lsi**}, in the spirit of \cite[Theorem~3]{ELS20}.

Formula \eqref{eq:entropy-ident} becomes more tractable when specified to specific kinds of diffeomorphisms. For instance, when $T$ is a product map of the form $T(x)=(\tau_1(x_1),\ldots,\tau_n(x_n))$, we get
\begin{equation} \label{eq:mT-diag}
\HH(\mu\|\lambda_n)-\HH(\gamma_n\|\lambda_n) \leq \inf_{\tau_1,\ldots,\tau_n\in\mr{Diff}(\R)}  \sum_{i=1}^n\int_{\R^n} \left\{\frac{1}{2}\left| \frac{(\tau_i'\circ\tau_i^{-1}) \partial_i f}{f} +( \log |\tau_i'\circ\tau_i^{-1}|)' \right|^2 - \log|\tau_i'\circ\tau_i^{-1}| \right\} \,\diff\mu.
\end{equation}
A similar simplified formula can be derived if $T$ is a rotationally invariant map of the form $T(x) = \sigma(|x|) x$. The equality cases of Theorem \ref{thm:mT} show that if $\mu$ is a product measure or a rotationally invariant measure, then the inequalities obtained by optimizing over the corresponding class of nonlinear transformations become equalities. For the case of a general probability measure $\mu$, we pose the following question.

\begin{question}
Let $\mu$ be an arbitrary absolutely continuous probability measure on $\R^n$. For which collection of diffeomorphisms $\tau_1,\ldots,\tau_n\in\mr{Diff}(\R)$ is the infimum \eqref{eq:mT-diag} attained?
\end{question}

A similar question can be asked for the optimal rotationally invariant change of variables.

We have not investigated whether nonlinear changes of variables may give rise to variational formulas \`a la \eqref{eq:entropy-ident} when applied to other estimates like the Gagliardo--Nirenberg--Sobolev inequality \eqref{eq:gns-gen} or Beckner's inequality \eqref{eq:beckner}.




\section{Tensorization of nonlinear logarithmic Sobolev inequalities in product spaces}
\label{sec:Hamming}
Let $I$ be a countable set,  $\{(\mb X_i, \mu_i)\}_{i\in I}$ a family of probability spaces where $\mb X_i$ is countable and denote their product space by $(\mb X,\mu) = (\prod_{i\in I} \mb X_i, \otimes_{i\in I}\mu_i)$.  For a point $x=(x_i)_{i\in I} \in \mb{X}$ and $i\in I$, we shall denote by $x_{\sim  i}$ the point $(x_j)_{j\neq i} \in \prod_{j\neq i} \mb{X}_j$ and by $\mu_{\sim i} \eqdef \otimes_{j\neq i} \mu_j$.  Moreover, for a point $z\in\prod_{j\neq i} \mb{X}_j$ and a function $f:\mb{X}\to\R$, we shall\mbox{ denote by $f_z:\mb{X}_i\to\R$ the restriction of $f$ given by}
\begin{equation}
\forall \ y\in\mb{X}_i,\qquad f_z(y) \eqdef f(z,y).
\end{equation} 
For each $i\in I$,  let $B_i$ be a functional acting on measurable functions $g:\mb \prod_{j\in J} \mb{X}_j\to \R$ for any $J\subseteq I$.  We shall say that the family of functionals $\{B_i\}_{i\in I}$ \emph{disintegrates} if it satisfies the identities
\begin{equation}
\forall \ i\in I, \qquad \int B_i\big(f_{x_{\sim i}}\big)\,\diff\mu_{\sim i}(x_{\sim i}) = B_i(f).
\end{equation}

Our main tensorization principle for nonlinear entropy inequalities is the following. 

\begin{theorem} \label{thm:tensor}
Fix a countable set $I$ and two collections of functionals $\{Q_i\}_{i\in I}$, $\{M_i\}_{i\in I}$ which disintegrate in the above sense.  Let $\Phi:\R\to \R$ be a concave function and suppose that, for any $i\in I$, every function $f_i:\mathbb X_i\to \R_{+}$ satisfies the inequality
\begin{equation} \label{eq:tensor-assumption}
\mr{Ent}_{\mu_i}[f_i]\le Q_i(f_i)+\mb{E}_{\mu_i}[f_i]\Phi\left(\frac{M_i(f_i)-Q_i(f_i)}{\mb{E}_{\mu_i}[f_i]}\right).
\end{equation}
Then, every function $f:\mathbb X\to \R_{+}$ satisfies
\begin{equation}
\mr{Ent}_\mu[f]\le \sum_{i\in I}Q_i(f)+\mb{E}_\mu[f]\sum_{i\in I}\Phi\left(\frac{M_i(f)-Q_i(f)}{\mb{E}_\mu[f]}\right). 
\end{equation}
\end{theorem}

\begin{proof}
Combining the subadditivity of entropy and the assumptions of the theorem (including the disintegration of $\{Q_i\}_{i\in I}$) we get that, for every $f:\mathbb X\to \R_{+}$,
\begin{equation}
\begin{split}
\mr{Ent}_\mu[f] &\le \sum_{i\in I}\int \mr{Ent}_{\mu_i}[f_{x_{\sim i}}]\,\diff\mu_{\sim i}(x_{\sim i})\\
&\stackrel{\eqref{eq:tensor-assumption}}{\leq}\sum_{i\in I} \int \left [Q_i(f_{x_{\sim i}})+\mb{E}_{\mu_i}[f_{x_{\sim i}}]\Phi\left(\frac{M_i(f_{x_{\sim i}})-Q_i(f_{x_{\sim i}})}{\mb{E}_{\mu_i}[f_{x_{\sim i}}]}\right)\right] \,\diff\mu_{\sim i}(x_{\sim i})\\
&= \sum_{i\in I}Q_i(f)+ \mb{E}_{\mu}[f] \sum_{i\in I} \int \Phi\left(\frac{M_i(f_{x_{\sim i}})-Q_i(f_{x_{\sim i}})}{\mb{E}_{\mu_i}[f_{x_{\sim i}}]}\right)\frac{\mb{E}_{\mu_i}[f_{x_{\sim i}}]}{\mb{E}_{\mu}[f]} \,\diff\mu_{\sim i}(x_{\sim i}).
\end{split}
\end{equation}
Since $\int \mb{E}_{\mu_i}[f_{x_{\sim i}}]\,\diff\mu_{\sim i}(x_{\sim i})= \mb{E}_{\mu}[f]$, and $\mb{E}_{\mu_i}[f_{x_{\sim i}}]\ge 0$, it follows that $\frac{\mb{E}_{\mu_i}[f_{x_{\sim i}}]}{\mb{E}_{\mu}[f]}\,\diff\mu_{\sim i}(x_{\sim i})$ defines a probability measure on $\prod_{j\neq i}\mathbb X_j$. Hence, by Jensen's inequality and disintegration, we get
\begin{equation}
\begin{split}
\mr{Ent}_\mu[f]&\le \sum_{i\in I}Q_i(f)+ \mb{E}_{\mu}[f]\sum_{i\in I} \Phi\left(\int \frac{M_i(f_{x_{\sim i}})-Q_i(f_{x_{\sim i}})}{\mb{E}_{\mu_i}[f_{x_{\sim i}}]}\cdot\frac{\mb{E}_{\mu_i}[f_{x_{\sim i}}]}{\mb{E}_{\mu}[f]} \,\diff\mu_{\sim i}(x_{\sim i})\right)\\
&= \sum_{i\in I}Q_i(f)+ \mb{E}_{\mu}[f] \sum_{i\in I} \Phi\left(\frac{M_i(f)-Q_i(f)}{\mb{E}_{\mu}[f]}\right).
\end{split}
\end{equation}
This completes the proof of the theorem.
\end{proof}

\begin{remark}
While Theorem \ref{thm:tensor} is stated in a general form which contains the disintegrating additive errors $\{Q_i\}_{i\in I}$, in its main application (Theorem \ref{thm:tensor-intro}) which refines the result of \cite{PS19}, these are assumed to be vanishing.   We chose to include the deficits in the general formulation above as such terms often appear in modified logarithmic Sobolev-type inequalities, especially in discrete settings (see, for instance, \cite{BL98, Wu00, BT06, Joh17}).
\end{remark}

\begin{proof} [Proof of Theorem \ref{thm:tensor-intro}]
The conclusion \eqref{eq:tensor-intro-conclusion} directly follows from Theorem \ref{thm:tensor} with $Q_i(f)=0$ and $M_i(f) = \mb{E}_{\pi^n} [\ms{E}_i(f^{1/p},f^{1-1/p})]$ since these functionals disintegrate.
\end{proof}

\begin{remark}
A different refinement of the log-Sobolev inequality on the discrete cube in terms of the logarithmic Laplace transform of the underlying measure can be found in \cite[Equation (12)]{Aug21}
\end{remark}


\part{Space forms: stochastic methods}
\label{part:space_forms}


\section{Preliminaries}

In this section we will introduce the necessary prerequisites from stochastic calculus on manifolds required to prove Theorem \ref{thm:manifolds}. We will be following the standard notation of \cite{Hsu02,Pet16}.


\subsection{The frame bundle}

Let $(\MM,\GG)$ be a complete $n$-dimensional Riemannian manifold. The \emph{orthonormal frame bundle} $\ms{O}(\MM)$ of $\MM$ is the set of all pairs of the form $(x, {\bf u})$, where $x\in\MM$ and ${\bf u} : \R^n \to T_x\MM$ is a Euclidean isometry. We shall denote by $\pi:\ms{O}(\MM)\to\MM$ the natural projection given by $\pi(x,{\bf u}) = x$. Any scalar-valued function $f:\MM\to\R$ admits a natural \emph{lift} ${\li f}:\ms{O}(\MM)\to\R$ given by
\begin{equation}
\forall \ (x,{\bf u})\in\ms{O}(\MM),\qquad {\li f}(x,{\bf u}) = f(x).
\end{equation}
Abusing notation, we shall often identify the pair $(x,{\bf u}) \in\ms{O}(\MM)$ with the isomorphism ${\bf u}$. 

A curve $\{ {\bf u}_t\}_{t\in[0,1]}$ in $\ms{O}(\MM)$ is called \emph{horizontal} if for every $a\in\R^n$, the vector field $\{ {\bf u}_ta \}_{t\in[0,1]}$ is parallel along the curve $\{ \pi {\bf u}_t\}_{t\in[0,1]}$ in $\MM$. A tangent vector ${\bf X}\in T_{{\bf u}}\ms{O}(\MM)$ is called horizontal if it is the tangent vector of a horizontal curve passing from ${\bf u}$. For any vector $X\in T_{\pi{\bf u}}\MM$ there exists a unique horizontal vector ${\bf X}\in T_{{\bf u}}\ms{O}(\MM)$ such that $\pi_\ast {\bf X} = X$; we say that ${\bf X}$ is the horizontal lift of $X$ at ${\bf u}$. Let $\{e_1,\ldots,e_n\}$ be the standard basis of $\R^n$. The $i$-th fundamental \emph{horizontal vector field} $\HH_i$ evaluated at a point ${\bf u}\in\ms{O}(\MM)$ is the horizontal lift of the vector ${\bf u}e_i \in T_{\pi{\bf u}}\MM$. Thus, for any $i\in\{1,\ldots,n\}$, the lift ${\li f}$ of a function $f:\MM\to\R$ satisfies
\begin{equation} \label{eq:laplacian-ident}
\forall \ {\bf u}\in\ms{O}(\MM),\qquad \HH_i {\li f}({\bf u}) = \nabla_{{\bf u}e_i} f(\pi {\bf u}).
\end{equation}
A vector field on $\ms{O}(\MM)$ is called \emph{horizontal} if it lies in the span of $\{\HH_1,\ldots,\HH_n\}$. We denote by $\langle\cdot,\cdot\rangle_{\mr{hor}}$ the natural inner product on the space of horizontal vector fields on $\ms{O}(\MM)$ given by
\begin{equation}
\left\langle \sum_{i=1}^n Z_i\HH_i, \sum_{i=1}^n W_i \HH_i \right\rangle_{\mr{hor}} = \sum_{i=1}^n Z_i W_i.
\end{equation} 
Moreover, we shall denote by $\nabla_{\mr{hor}}{\li f} = (\HH_1{\li f},\ldots\HH_n {\li f})\in\R^n$ the \emph{horizontal gradient} of a given function ${\li f}:\ms{O}(\MM)\to\R$. The frame bundle $\ms{O}(\MM)$ is equipped with Bochner’s horizontal Laplacian
\begin{equation}
\label{eq:Laplace_def}
\Delta_{\ms{O}(\MM)}\eqdef\sum_{i=1}^nH_i^2,
\end{equation}
and can be verified (see \cite[Proposition~3.1.2]{Hsu02}) that the lift $\li{f}$ of any function $f:\MM\to\R$ satisfies
\begin{equation} \label{eq:laplacian-ident}
\forall \ {\bf u}\in\ms{O}(\MM),\qquad \Delta_{\ms{O}(\MM)} {\li f}({\bf u}) = \Delta f(\pi {\bf u}),
\end{equation}
where $\Delta$ is the Laplace--Beltrami operator of $(\MM,\GG)$.

We record for future reference the following very useful expression for the action of the commutator of $\Delta_{\ms{O}(\MM)}$ with $\HH_i$ on lifted functions.

\begin{lemma} \label{lem:ricci}
If $f:(\MM,\GG)\to\R$ is a smooth function, then for any $i\in\{1,\ldots,n\}$, its lift ${\li f}$ satisfies
\begin{equation} \label{eq:commute}
\forall \ {\bf u}\in\ms{O}(\MM),\qquad \Delta_{\ms{O}(\MM)} \HH_i {\li f}({\bf u}) - \HH_i \Delta_{\ms{O}(\MM)} {\li f}({\bf u}) = \mathrm{Ric}(\nabla f, {\bf u}e_i)(\pi {\bf u}),
\end{equation}
where $\mr{Ric}(\cdot,\cdot)$ is the Ricci tensor on $\MM$.
\end{lemma}

\begin{proof}
We shall follow the notation of \cite[Section~5.5]{Hsu02}. For $i,k\in\{1,\ldots,n\}$, it follows from \cite[Lemma~5.5.1]{Hsu02} that the commutator $[\HH_i,\HH_k]$ is a vertical vector field and thus $[\HH_i,\HH_k]{\li f}=0$, i.e.
\begin{equation}
\HH_k\HH_i {\li f} = \HH_i \HH_k{\li f}.
\end{equation}
Therefore, we have
\begin{equation} \label{eq:com1}
\HH_k^2 \HH_i {\li f} = \HH_k \HH_i \HH_k {\li f} = [\HH_k,\HH_i] \HH_k {\li f} + \HH_i \HH_k^2 {\li f}.
\end{equation}
Substituting the expression of \cite[Lemma~5.5.1]{Hsu02} for $[\HH_k,\HH_i]$, we get
\begin{equation} \label{eq:com2}
[\HH_k,\HH_i] \HH_k {\li f} = -\sum_{a,b} \Omega_{ki}^{ab} V_{ab} \HH_k {\li f} = - \sum_{a,b} \Omega_{ki}^{ab} [V_{ab}, \HH_k] {\li f},
\end{equation} 
where in the last identity we used that $V_{ab}{\li f}=0$. Again, by \cite[Lemma~5.5.1]{Hsu02}, if we denote by $A_{ab}^{k\ell}$ the number $\tfrac{1}{2}$ for $(a,b)=(k,\ell)$ and $-\tfrac{1}{2}$ for $(a,b)=(\ell,k)$, and zero otherwise, we obtain
\begin{equation} \label{eq:com3}
-\sum_{a,b} \Omega_{ki}^{ab} [V_{ab}, \HH_k] {\li f} = -\sum_{a,b,\ell} \Omega_{ki}^{ab} A_{ab}^{k\ell} \HH_\ell {\li f} = \frac{1}{2} \sum_{\ell} \big\{ \Omega_{ki}^{\ell k}-\Omega_{ki}^{k\ell}\big\} \HH_\ell{\li f} = \sum_\ell \Omega_{ki}^{\ell k} \HH_\ell {\li f},
\end{equation}
where the antisymmetry of $\Omega$ on the top indices follows from its definition in \cite[p.~153]{Hsu02} as it is an $\mathfrak{o}(d)$-valued tensor. Combining \eqref{eq:com1}, \eqref{eq:com2}, \eqref{eq:com3} and summing over $k$, we deduce that
\begin{equation}
\Delta_{\ms{O}(\MM)} \HH_i {\li f}({\bf u}) - \HH_i \Delta_{\ms{O}(\MM)} {\li f}({\bf u}) = \sum_{k,\ell} \Omega_{ki}^{\ell k} \HH_\ell {\li f}({\bf u}) = \sum_{k,\ell} \Omega_{ki}^{\ell k} \nabla_{{\bf u}e_\ell} f(\pi {\bf u}).
\end{equation}
Now, observe that by the definition of $\Omega$ in terms of the Riemann tensor $\msf{R}$ of $\MM$ in \cite[p.~149]{Hsu02},
\begin{equation}
\Omega_{ki}^{\ell k} = \GG\big( \msf{R}({\bf u}e_k, {\bf u}e_i) {\bf u}e_k, {\bf u}e_\ell\big)
\end{equation}
and the conclusion follows from the definition of Ricci curvature.
\end{proof}


\subsection{Brownian motion on manifolds} Let $W_t=(W_t^1,\ldots,W_t^n)$ be a standard Brownian motion on $\R^n$ and $(\MM, \GG)$ be a complete $n$-dimensional Riemannian manifold. We consider the following stochastic differential equation on the frame bundle $\ms{O}(\MM)$,
\begin{equation} \label{eq:bm}
\diff\Phi_t = \sum_{i=1}^n \HH_i (\Phi_t) \circ \diff W_t^i,
\end{equation}
where the shorthand notation $\circ$ refers to the Stratonovitch integral. In It\^o terms, the above SDE asserts that for every smooth $g:\ms{O}(\MM)\to\R$, we have
\begin{equation} \label{eq:sde-phi}
\diff g(\Phi_t) = \sum_{i=1}^n \HH_ig(\Phi_t)\, \diff W_t^i  + \frac{1}{2} \Delta_{\ms{O}(\MM)}g(\Phi_t)\, \diff t.
\end{equation}
For any initial condition $\Phi_0 = {\bf u}\in\ms{O}(\MM)$, this equation has a strong solution which does not blow up in finite time if the Ricci curvature of $\MM$ is bounded from below by any constant $\kappa\in\R$ (see \cite[Theorem~4.2.4]{Hsu02} and \cite{Var83} for a sufficient and almost necessary condition for stochastic completeness). We denote by $B_t = \pi\Phi_t$, where $t\geq0$, the {\it Brownian motion} on $\MM$ whose starting point is $x=\pi{\bf u}\in \MM$. Applying \eqref{eq:sde-phi}, we deduce that for any smooth function $f:\MM\to\R$, the Brownian motion $\{B_t\}_{t\geq0}$ satisfies the SDE
\begin{equation}
\diff f(B_t) = \sum_{i=1}^n \Phi_t^i f(B_t)\, \diff W_t^i + \frac{1}{2} \Delta f(B_t)\,\diff t.
\end{equation}


\subsection{The F\"ollmer process and Lehec's formula} In this section we introduce an analogue of the classical F\"ollmer process \cite{Sch32,Fol86,Leo14} on Riemannian manifolds (see also \cite[Section~5.4]{Hsu02}). We then present a result of Lehec \cite{Leh17} who used this process to give a stochastic proof of the dimensional logarithmic Sobolev inequality for manifolds with Ricci curvature bounded below (see \cite{BL06,BBL17} for more general statements proven via semigroup arguments). 

Let $W_t=(W_t^1,\ldots,W_t^n)$ be a standard Brownian motion on $\R^n$ and $(\MM, \GG)$ be a complete $n$-dimensional Riemannian manifold whose Ricci curvature is bounded from below. We shall denote by $\diff x$ the volume measure on $\MM$ and by $\{P_t\}_{t\geq0}$ the heat semigroup on $\MM$. Recall that for a smooth function $h:\MM\to\R$, the action of the heat flow $\{P_t\}_{t\geq0}$ on $g$ is characterized by the ordinary differential equation
\begin{equation}
\forall \ t>0, \ x\in\MM,\qquad \frac{\partial P_th}{\partial t}(x) = \frac{1}{2}\Delta P_t h(x)
\end{equation}
with initial condition $P_0h = h$ on $\MM$. We recall that the heat semigroup and the Laplacian commute: $ \Delta P_th=P_t\Delta h$, and we write $P_t\nabla^2 f(x)$ for the 2-tensor on $T_x\MM$ identified with the symmetric matrix $( P_t\nabla^2 f(\Phi_0e_i, \Phi_0e_j)(x))_{i,j=1}^n$. Note that $P_t$ and\mbox{ $\nabla^2$ \emph{do not} commute (cf. Theorem \ref{thm:wang}).}

For a positive function $f:\MM\to\R_+$ and $T>0$, we consider the following system of stochastic differential equations with respect to $(\Psi_t, X_t) \in \ms{O}(\MM)\times\MM$
\begin{equation} \label{eq:def-follmer}
\begin{cases}
\diff\Psi_t = \sum_{i=1}^n \HH_i (\Psi_t) \circ \big( \diff W_t^i + \Psi_t^{-1} \nabla\log P_{T-t}f(X_t) \, \diff t\big)\\
X_t = \pi\Psi_t
\end{cases},
\end{equation}
where the notation $\circ$ again refers to the Stratonovitch integral. It is known (see \cite[Theorem~7]{Leh17}) that if $f$ is a smooth-enough positive function, then for any initial condition $\Psi_0 = {\bf u}\in\ms{O}(\MM)$, the system \eqref{eq:def-follmer} has a strong solution on $[0,T]$. In \cite[Theorem~7]{Leh17}, Lehec proved the  manifold version of an important representation formula for relative entropy in terms of the {\it F\"ollmer process} $X_t$, first proven in their earlier work \cite{Leh13}.

\begin{theorem} [Lehec]
\label{thm:Lehec}
Let $(\MM, \GG)$ be a complete $n$-dimensional Riemannian manifold whose Ricci curvature is bounded from below and fix a smooth enough positive density function $f:\MM\to\R_+$ and $T>0$. If $\{X_t\}_{t\in[0,T]}$ is a solution of \eqref{eq:def-follmer} with initial condition $\Psi_0 = {\bf u}$ and $\pi{\bf u}=x$, then the relative entropy of the measure $\mu$ with density $\frac{\diff\mu}{\diff P_T\delta_x} = f$ is
\begin{equation} \label{eq:lehec-formula}
\HH(\mu \| P_T\delta_x) = \frac{1}{2}\mb{E} \left[ \int_0^T \big| \nabla \log P_{T-t}f(X_t)\big|^2\,\diff t \right].
\end{equation}
where $|v|\eqdef \GG_x(v,v)$ for $x\in \MM$ and $v\in T_x\MM$.
\end{theorem}

It is worth pointing out that, in view of the decay and regularity of the heat kernel on space forms (see, e.g.,~\cite[Chapter~6]{Cha84} and \cite{McK70,DGM76,GN98}),  it suffices to assume that the functions for which we wish to prove the logarithmic Sobolev inequalities of Theorem \ref{thm:manifolds}  are Lipschitz and bounded away from 0.  Therefore,  the regularity conditions required for the function $f$ in Lehec's theorem will always be tacitly assumed to hold.

We record for future reference the following  computations (see also \cite[Equations (5.5.2) -- (5.5.4)]{Hsu02}) on the SDE satisfied by partial derivatives of the logarithm of the heat kernel.

\begin{lemma} \label{lem:sde-for-Hi}
Let  $(\MM, \GG)$ be a complete $n$-dimensional Riemannian manifold and fix a smooth enough positive density function $f:\MM\to\R_+$ and $T>0$. Denote by ${F}_t:\MM\to\R$ the function given by
\begin{equation}
\forall \ x\in\MM,\qquad { F}_t(x) = \log P_{T-t} f(x)
\end{equation} 
and by ${\li F}_t$ the lift of $F_t$ onto $\ms{O}(\MM)$. If $\{X_t\}_{t\in[0,T]}$ is a solution of \eqref{eq:def-follmer} and $\{W_t\}_{t\geq0}$ is a standard Brownian motion on $\R^n$, then for every $i\in\{1,\ldots,n\}$ we have
\begin{equation}
\diff \HH_i{\li F}_t (\Psi_t) = \big\langle \nabla^{\mr{hor}} \HH_i {\li F}_t(\Psi_t), \,\diff W_t \big\rangle + \frac{1}{2} \mr{Ric}\big(\nabla F_t , \Psi_t e_i\big) (X_t)\,\diff t.
\end{equation}
\end{lemma}

\begin{proof}
Using It\^o's formula and \eqref{eq:def-follmer}, we get (omitting the dependence on $\Psi_t$ on the right-hand side of \eqref{eq:use-ito})
\begin{equation} \label{eq:use-ito}
\diff \HH_i{\li F}_t (\Psi_t) = \big\langle \nabla^{\mr{hor}} \HH_i {\li F}_t, \,\diff W_t \big\rangle + \left\{ \frac{\partial \HH_i {\li F}_t}{\partial t} + \frac{1}{2} \Delta_{\ms{O}(\MM)} \HH_i {\li F}_t + \langle \nabla^{\mr{hor}} \HH_i {\li F}_t, \nabla^{\mr{hor}} {\li F}_t\rangle \right\} \,\diff t.
\end{equation}
Observe that the function ${F}_t$ satisfies the equation
\begin{equation}
\frac{\partial {\li F}_t}{\partial t} = -\frac{1}{2}\Delta_{\ms{O}(\MM)} {\li F}_t - \frac{1}{2}\big|\nabla^{\mathrm{hor}} {\li F}_t\big|^2,
\end{equation}
which, after applying $\HH_i$ on both sides, gives 
\begin{equation} \label{eq:use-heat}
\frac{\partial \HH_i{\li F}_t}{\partial t} = \HH_i \frac{\partial {\li F}_t}{\partial t} = -\frac{1}{2} \HH_i\Delta_{\ms{O}(\MM)} {\li F}_t - \frac{1}{2} \HH_i \big| \nabla^{\mr{hor}} {\li F}_t \big|^2.
\end{equation}
Moreover, we have
\begin{equation} \label{eq:use-lc}
 \HH_i \big| \nabla^{\mr{hor}} {\li F}_t \big|^2 =  \HH_i \langle \nabla^{\mr{hor}} {\li F}_t,  \nabla^{\mr{hor}} {\li F}_t \rangle = 2 \langle \HH_i \nabla^{\mr{hor}} {\li F}_t,  \nabla^{\mr{hor}} {\li F}_t \rangle  = 2\langle \nabla^{\mr{hor}}\HH_i {\li F}_t,  \nabla^{\mr{hor}} {\li F}_t \rangle,
\end{equation}
where in the last identity we use that $[\HH_i,\HH_k] {\li h} = 0$ for any lifted function ${\li h}$ on $\ms{O}(\MM)$ \cite[Lemma~5.5.1]{Hsu02}. Substituting \eqref{eq:use-heat} and \eqref{eq:use-lc} in \eqref{eq:use-ito}, we finally obtain
\begin{equation}
\diff \HH_i{\li F}_t (\Psi_t) = \big\langle \nabla^{\mr{hor}} \HH_i {\li F}_t(\Psi_t), \,\diff W_t \big\rangle + \frac{1}{2} [\Delta_{\ms{O}(\MM)}, \HH_i]{\li F}_t(\Psi_t) \,\diff t
\end{equation}
and the desired identity follows immediately from Lemma \ref{lem:ricci}.
\end{proof}


\subsection{The heat flow on space forms} 
The classical Bochner formula  (see, e.g., \cite{Wan19}) implies that if $(\MM,\GG)$ is a Riemannian manifold with constant Ricci curvature $\mr{Ric} = \kappa\in\R$, then
\begin{equation}
\forall \ t\geq0,\qquad  \nabla P_t f = e^{-\kappa t/2} P_t \nabla f
\end{equation}
for every smooth function $f:\MM\to\R$. In \cite{Wan19}, Wang investigated commutation relations of this form for second order derivatives instead of the gradient $\nabla$. We shall use the following result.

\begin{theorem} [Wang] \label{thm:wang}
A Riemannian manifold $(\MM,\GG)$ of dimension $n$ has constant sectional curvature $\kappa\in\R$ if and only if the Hessian tensor of every smooth function $f:\MM\to\R$ satisfies
\begin{equation} \label{eq:wang}
\forall \ r\geq0, \qquad \nabla^2P_rf = e^{-n\kappa r} P_r\nabla^2f + \frac{1-e^{-n\kappa r}}{n} P_r\Delta f\cdot \GG.
\end{equation}
\end{theorem}


\section{Intrinsic dimensional logarithmic Sobolev inequality in space forms}
\label{sec:LSImanifold}

Having explained the necessary background we can now present  Theorem \ref{thm:manifolds}. We first recall that when $\frac{\diff\mu}{\diff P_T\delta_x}=f$, we have
\begin{equation*}
\HH(\mu \| P_T\delta_x) =P_T(f\log f)(x)-P_Tf(x)\log P_Tf(x).
\end{equation*}
\begin{theorem} \label{thm:manifolds_sec}
Let $(\MM,\GG)$ be an $n$-dimensional Riemannian manifold with constant sectional curvature $\kappa\in\R\setminus\{0\}$ with the associated  heat semigroup $\{P_t\}_{t\ge 0}$. Fix $T>0$, $x\in\MM$, a smooth positive function $f:\MM\to\R$ with $\int_M f\,\diff P_T\delta_x=1$, and let $\mu$ be the probability measure with $\frac{\diff\mu}{\diff P_T\delta_x}=f$. Define the 2-tensor $C(t) = \frac{e^{n\kappa t}}{n\kappa}A+tB$ for $t\in\R$, where  $A, B$ are 2-tensors given by
\begin{equation}
\begin{cases}
A =- e^{-n\kappa T}  \big(P_T\nabla^2 f(x) - \frac{1}{n} P_T \Delta f(x)\cdot \GG\big) \\ B = \Big(\frac{(n-1)\kappa}{2} - \frac{\Delta P_T f(x)}{n}\Big) \cdot \GG,
\end{cases}
\end{equation}
and let $\mb{E}_\mu \big(\nabla \log  f\big)^{\otimes 2}\eqdef  \mb{E}\big[ \big( \nabla^{\mr{hor}} \log {\li f}(\Psi_T)\big)^{\otimes 2}\big] $.  Then, we have the local intrinsic dimensional logarithmic Sobolev inequality
\begin{align} \label{eq:manifolds-lsi_sec}
\begin{split}
&\HH(\mu \| P_T\delta_x)\\
& \leq \frac{1}{2} \int_0^T\!\!\!  \mr{tr}\Big[e^{C(t)-C(T)} \Big(\GG + \mb{E}_\mu \big(\nabla \log f\big)^{\otimes 2}\!\! \int_t^T e^{2C(s)-2C(T)} \,\diff s \Big)^{-1}\!\! \mb{E}_\mu \big(\nabla \log f\big)^{\otimes 2}  e^{C(t)-C(T)}\Big]\,\diff t,
\end{split}
\end{align}
and the reverse local intrinsic dimensional logarithmic Sobolev inequality
\begin{align} \label{eq:manifolds-rev-lsi_sec}
\begin{split}
&\HH(\mu \| P_T\delta_x)\\
& \geq \frac{1}{2} \int_0^T \mr{tr}\Big[e^{C(t)-C(0)} \Big( \GG -\big( \nabla \log P_Tf(x)\big)^{\otimes 2} \int_0^t e^{2C(s)-2C(0)}\,\diff s\Big)^{-1} \big( \nabla \log P_Tf(x)\big)^{\otimes 2} e^{C(t)-C(0)}\Big] \,\diff t.
\end{split}
\end{align}
\end{theorem}

The proof of Theorem \ref{thm:manifolds_sec} (see also the stronger Theorem \ref{thm:manifolds_opt}) is modeled after the stochastic proof by Eldan, Lehec, and Shenfeld \cite{ELS20} of the intrinsic dimensional logarithmic Sobolev inequality in flat space \eqref{eq:lsi**} (and a weaker reverse inequality \cite[Theorem 3]{ELS20}). A basic ingredient of this approach is deriving a stochastic differential equation for the \emph{tensor} whose trace is the term $ \big| \nabla \log P_{T-t}f(X_t)\big|^2$ in \eqref{eq:lehec-formula}. This is the content of the next lemma for which we establish the following notation. Let $\{B_t\}_{t\geq0}$ be a Brownian motion on $\MM$ with $B_0=x$.  As before, we denote by $F_t$ the function $\log P_{T-t}f$ and by ${\li F_t}$ its horizontal lift on $\ms{O}(\MM)$. Moreover, we shall denote by $G_t$ the function $\exp F_t=P_{T-t}f$ and by ${\li G}_t = \exp {\li F}_t$ its lift. Consider the random matrices $\msf{Q}(t),\msf{P}(t)\in\mathbb M_n(\R)$ (the space of $n\times n$ square matrices over $\R$) given by
\begin{align}\label{eq:Qdef}
\begin{split}
&\msf{Q}_{ij}(t) \eqdef \HH_i\HH_j {\li F_t}(\Psi_t) = \HH_j\HH_i {\li F_t}(\Psi_t) = \msf{Q}_{ji}(t),\\
&\msf{P}(t)\eqdef\msf{Q}(t)^2.
\end{split}
\end{align}
We can now derive the aforementioned stochastic differential equation.

\begin{lemma} \label{lem:Q}
Let $(\MM,\GG)$ be a Riemannian manifold. In the terminology above, for every $i,j\in\{1,\ldots,n\}$, there exists a martingale $\{M_{ij}(t)\}_{t\in[0,T]}$ such that for $t\in[0,T]$, we have
\begin{equation*}
\HH_i{\li F_t}(\Psi_t) \cdot \HH_j{\li F}_t(\Psi_t) = M_{ij}(t) +\frac{1}{2} \int_0^t \mr{Ric}\big( \nabla F_s(X_s), \HH_i {\li F}_s(\Psi_s) \cdot \Psi_s e_j + \HH_j {\li F}_s(\Psi_s) \cdot \Psi_s e_i \big)\, \diff s + \int_0^t \msf{P}_{ij}(s) \, \diff s.
\end{equation*}
\end{lemma}

\begin{proof}
Observe that by the chain rule, we have (omitting the dependence on $\Psi_t$ on the right-hand side below)
\begin{equation} \label{eq:Q1}
\msf{Q}_{ij}(t) = \HH_i\HH_j {\li F}_t =  \HH_i\HH_j \log{\li G}_t = \frac{\HH_i\HH_j {\li G}_t}{{\li G}_t} - \HH_i {\li F}_t\cdot \HH_j{\li F}_t
\end{equation}
and by the definition and symmetry of the matrix $\msf{Q}(t)$,
\begin{equation} \label{eq:Q2}
\msf{P}_{ij}(t) = \sum_{k=1}^n \msf{Q}_{ik}(t) \msf{Q}_{kj}(t) = \sum_{k=1}^n \HH_k \HH_i {\li F}_t \cdot \HH_k\HH_j {\li F}_t = \left\langle \nabla^{\mr{hor}} \HH_i {\li F}_t, \nabla^{\mr{hor}}\HH_j{\li F}_t\right\rangle.
\end{equation}
Combining It\^o's product rule with Lemma \ref{lem:sde-for-Hi}, we get that for $i,j\in\{1,\ldots,n\}$,
\begin{equation} \label{eq:Q3}
\begin{split}
\diff\big\{ \HH_i{\li F_t}(\Psi_t) \cdot &\HH_j{\li F}_t(\Psi_t)\big\}  = \HH_i {\li F}_t(\Psi_t) \diff \HH_j {\li F}_t(\Psi_t) + \HH_j {\li F}_t(\Psi_t) \diff \HH_i {\li F}_t(\Psi_t) + \diff\HH_j {\li F}_t(\Psi_t) \cdot \diff \HH_i {\li F}_t(\Psi_t)
\\ & = \left\{ \frac{1}{2}\mr{Ric}\big( \nabla F_t, \HH_i {\li F}_t \cdot \Psi_t e_j + \HH_j {\li F}_t \cdot \Psi_te_i \big) + \left\langle \nabla^{\mr{hor}} \HH_i {\li F}_t, \nabla^{\mr{hor}}\HH_j{\li F}_t\right\rangle \right\}\,\diff t
\\ & \qquad \qquad \qquad \qquad \qquad \qquad  \qquad \quad \quad +  \left\langle \HH_i{\li F}_t \nabla^{\mr{hor}}\HH_j{\li F}_t + \HH_j {\li F}_t \nabla^{\mr{hor}}\HH_i{\li F}_t, \,\diff W_t\right\rangle,
\end{split}
\end{equation}
where in the right-hand side we again omitted the dependence on $\Psi_t$ and $X_t$. Denoting the term in the last line by $\diff \MM_{ij}(t)$, it is clear that $\{M_{ij}(t)\}_{t\in[0,T]}$ is a martingale and \eqref{eq:Q3} becomes
\begin{equation*} \label{eq:Q4}
\HH_i{\li F_t}(\Psi_t) \cdot \HH_j{\li F}_t(\Psi_t) = \MM_{ij}(t) +\frac{1}{2} \int_0^t \mr{Ric}\big( \nabla F_s(X_s), \HH_i {\li F}_s(\Psi_s) \cdot \Psi_s e_j + \HH_j {\li F}_s(\Psi_s) \cdot \Psi_s e_i \big)\, \diff s + \int_0^t \msf{P}_{ij}(s) \, \diff s,
\end{equation*}
where we also used \eqref{eq:Q2}. This is the desired identity.
\end{proof}

The stochastic differential equation of Lemma \ref{lem:Q} will allow us to derive a differential equation for 
\begin{equation}
\label{eq:vij}
\forall \ t\in(0,T),\qquad v_{ij}(t) \eqdef \mb{E}\big[ \HH_i {\li F}_t(\Psi_t) \cdot \HH_j {\li F}_t(\Psi_t)\big],\qquad i,j\in\{1,\ldots,n\};
\end{equation}
note that with this notation, \eqref{eq:lehec-formula} reads $\HH(\mu \| P_T\delta_x)  = \frac{1}{2}\int_0^T \mr{tr}\big[ v(t)\big] \,\diff t$. We will then turn the differential equation into a differential \emph{inequality} from which Theorem \ref{thm:manifolds_opt} and Theorem \ref{thm:manifolds_sec}  shall follow. To derive the  differential equation for $v(t)$ we start by defining
\begin{equation}
\label{eq:mn}
m(t) \eqdef \mb{E}[-\msf{Q}(t)]\quad \text{and}\quad n(t) \eqdef \mb{E}[\msf{P}(t)]. 
\end{equation}
Assuming that the underlying manifold $\MM$ is Einstein and taking expectations, we deduce the following differential equation for $v(t)$. 

\begin{lemma}\label{lem:v(t)ODE}
Let $(\MM,\GG)$ be an Einstein manifold with constant Ricci curvature $\mr{Ric}=\rho\GG$ for some $\rho\in\R$. For every $i,j\in\{1,\ldots,n\}$ and $t\in(0,T)$, we have
\begin{equation} \label{eq:for-einstein}
\frac{\diff v_{ij}(t)}{\diff t} = n_{ij}(t) + \rho v_{ij}(t).
\end{equation}
\end{lemma}

\begin{proof}
Since $\MM$ has constant Ricci curvature $\rho\GG$, we have
\begin{equation}
\mathrm{Ric}\big( \nabla F_s(X_s), \HH_i {\li F}_s(\Psi_s)\cdot \Psi_s e_j\big) = \rho \HH_i {\li F}_s(\Psi_s)\cdot \GG(\nabla F_s(X_s), \Psi_s e_j) = \rho\HH_i {\li F}_s(\Psi_s)\cdot \HH_j {\li F}_s(\Psi_s).
\end{equation}
Plugging this in the rightmost term of Lemma \ref{lem:Q}, we get that
\begin{equation}
\HH_i{\li F_t}(\Psi_t) \cdot \HH_j{\li F}_t(\Psi_t) = \MM_{ij}(t) +\rho \int_0^t \HH_i {\li F}_s(\Psi_s)\cdot \HH_j {\li F}_s(\Psi_s)\,\diff s + \int_0^t \msf{P}_{ij}(s) \, \diff s.
\end{equation}
The result follows after taking expectation (since $\mb{E} M_{ij}(t) = M_{ij}(0) = 0$) and differentiating.
\end{proof}
In order to turn \eqref{eq:for-einstein} into a differential \emph{inequality}  we will use Jensen's inequality $n(t)\succeq m(t)^2$ where we used $\msf{P}=\msf{Q}^2$. To use the latter inequality we need to better understand the term $m(t)$. On manifolds of constant curvature, $m(t)$ takes the following simple form. 

\begin{lemma}
Let $(\MM,\GG)$ be an $n$-dimensional Riemannian manifold with constant sectional curvature $\kappa \in\R$. For every $i,j\in\{1,\ldots,n\}$ and $t\in(0,T)$, we have
\begin{equation} \label{eq:from-v-to-m}
m_{ij}(t) = v_{ij}(t) - e^{-n\kappa(T-t)} P_T \nabla^2 f(\Phi_0e_i, \Phi_0e_j)(x) - \frac{1-e^{-n\kappa(T-t)}}{n} P_T\Delta f(x)\cdot \delta_{ij}.
\end{equation}
\end{lemma}

\begin{proof}
Taking expectations in \eqref{eq:Q1}, we obtain
\begin{equation} \label{eq:m00}
m_{ij}(t) = \mb{E}[-\msf{Q}_{ij}(t)] \stackrel{\eqref{eq:Q1}}{=} v_{ij}(t) -\mb{E}\left[ \frac{\HH_i\HH_j {\li G}_t(\Psi_t)}{{\li G}_t(\Psi_t)}\right].
\end{equation}
It follows from \eqref{eq:def-follmer} and \eqref{eq:bm} that $\Psi_t$ has law $f(B_T)$ with respect to $\Phi_t$ for every $t\in[0,T]$ (see also the proof of \cite[Theorem~7]{Leh17} for an argument based on Girsanov's theorem). Therefore, by the tower property of conditional expectation, we have
\begin{equation} \label{eq:change-measure}
\begin{split}
\mb{E}\left[ \frac{\HH_i\HH_j {\li G}_t(\Psi_t)}{{\li G}_t(\Psi_t)}\right] = \mb{E}\left[ \frac{\HH_i\HH_j {\li G}_t(\Phi_t)}{{\li G}_t(\Phi_t)} f(B_T) \right] &= \mb{E}\left[ \frac{\HH_i\HH_j {\li G}_t(\Phi_t)}{{\li G}_t(\Phi_t)} \mb{E}\big[f(B_T)\big| \{\Phi_r\}_{r\leq t}\big] \right] 
\\ & =  \mb{E}\left[ \frac{\HH_i\HH_j {\li G}_t(\Phi_t)}{{\li G}_t(\Phi_t)} P_{T-t}f(B_t)\right] = \mb{E}\big[ \HH_i\HH_j {\li G}_t(\Phi_t) \big].
\end{split}
\end{equation}
Recall that for any function $h:\MM\to\R$ with horizontal lift ${\li h}$, we have
\begin{equation}\label{eq:hess-ident}
\forall \ {\bf u}\in\ms{O}(\MM),\qquad \HH_i\HH_j {\li h} ( {\bf u}) = \nabla^2h({\bf u}e_i, {\bf u}e_j)(\pi {\bf u}),
\end{equation}
see, e.g., \cite[Equation~(2.2.3)]{Hsu02}. Combining \eqref{eq:hess-ident} with Theorem \ref{thm:wang}, we deduce that
\begin{equation}
\begin{split}
\HH_i\HH_j & {\li G}_t (\Phi_t)  \stackrel{\eqref{eq:hess-ident}}{=} \nabla^2 P_{T-t}f (\Phi_te_i, \Phi_te_j)(B_t)
\\ & \stackrel{\eqref{eq:wang}}{=} e^{-n\kappa(T-t)} P_{T-t}\nabla^2 f(\Phi_te_i, \Phi_te_j)(B_t) + \frac{1-e^{-n\kappa(T-t)}}{n} P_{T-t}\Delta f(B_t) \cdot \GG(\Phi_t e_i, \Phi_t e_j)
\\ & = e^{-n\kappa(T-t)} P_{T-t}\nabla^2 f(\Phi_te_i, \Phi_te_j)(B_t) + \frac{1-e^{-n\kappa(T-t)}}{n} P_{T-t}\Delta f(B_t) \cdot \delta_{ij},
\end{split}
\end{equation}
where in the last equality we used that $\{\Phi_te_1,\ldots,\Phi_te_n\}$ is an orthonormal basis of $T_{B_t}\MM$. Taking expectations on both sides, we get
\begin{equation}
\mb{E}\big[\HH_i\HH_j {\li G}_t (\Phi_t)  \big]  = e^{-n\kappa(T-t)} \mb{E}\big[ P_{T-t}\nabla^2 f(\Phi_te_i, \Phi_te_j)(B_t)  \big] + \frac{1-e^{-n\kappa(T-t)}}{n} \mb{E}\big[P_{T-t}\Delta f(B_t) \big]\cdot \delta_{ij}.
\end{equation}
By the definition \cite[Equation~(1.2)]{Wan19} of the action of $\{P_s\}_{s\geq0}$ on tensors, we have
\begin{equation*}
 \mb{E}\big[ P_{T-t}\nabla^2 f(\Phi_te_i, \Phi_te_j)(B_t)  \big]  =  \mb{E}\big[ \nabla^2 f(\Phi_Te_i, \Phi_Te_j)(B_T)  \big] = P_T \nabla^2 f(\Phi_0e_i, \Phi_0e_j)(x),
\end{equation*}
where the last identity follows from the definition of stochastic parallel transport given by $\{\Phi_s\circ\Phi_0^{-1}\}_{s\geq0}$ (see \cite[Section~2.3]{Hsu02}). Similarly, we have
\begin{equation}
\mb{E}\big[P_{T-t}\Delta f(B_t) \big] = \mb{E} \big[ \Delta f(B_T)\big] = P_T\Delta f(x)
\end{equation}
and combining everything we deduce that
\begin{equation} \label{eq:use-parallel-transport}
\mb{E}\big[\HH_i\HH_j {\li G}_t (\Phi_t)  \big]  = e^{-n\kappa(T-t)} P_T \nabla^2 f(\Phi_0e_i, \Phi_0e_j)(x) + \frac{1-e^{-n\kappa(T-t)}}{n} P_T\Delta f(x)\cdot \delta_{ij}.
\end{equation}
Plugging \eqref{eq:use-parallel-transport} and \eqref{eq:change-measure} in \eqref{eq:m00} completes the proof.
\end{proof}

We are now ready to derive the differential inequality for $v(t)$. For simplicity, we shall denote by $c_T \eqdef  P_T\Delta f(x)$ and by $J_T$ the symmetric matrix with
\begin{equation}
\label{eq:J_T}
(J_T)_{ij} \eqdef P_T\nabla^2 f(\Phi_0 e_i, \Phi_0 e_j)(x) - \frac{1}{n} P_T \Delta f(x)\cdot \delta_{ij},
\end{equation}
which satisfies $\mr{tr} J_T=0$. Combining all of the above, we get the following matrix inequality:

\begin{proposition} \label{prop:matrix-ineq}
Let $(\MM,\GG)$ be an $n$-dimensional Riemannian manifold with constant sectional curvature $\kappa \in\R$. For every $t\in(0,T)$, we have
\begin{align}
\label{eq:matrix-ineq_opt}
\begin{split}
\frac{\diff v(t)}{\diff t} \succeq v(t)^2- \Big( e^{-n\kappa(T-t)} J_T+\tfrac{c_T}{n}\cdot \msf{Id}_n  \Big) &v(t)  - v(t) \Big( e^{-n\kappa(T-t)} J_T+ \tfrac{c_T}{n}\cdot \msf{Id}_n  \Big)\\ 
& +  \Big( e^{-n\kappa(T-t)} J_T + \tfrac{c_T}{n}\cdot \msf{Id}_n \Big)^2 + (n-1)\kappa v(t),
\end{split}
\end{align}
so in particular,
\begin{equation}
\label{eq:matrix-ineq}
\frac{\diff v(t)}{\diff t} \succeq v(t)^2+ \Big( \big(\tfrac{(n-1)\kappa}{2}-\tfrac{c_T}{n}\big)\cdot \msf{Id}_n - e^{-n\kappa(T-t)} J_T \Big) v(t)  + v(t) \Big( \big(\tfrac{(n-1)\kappa}{2}-\tfrac{c_T}{n}\big)\cdot \msf{Id}_n - e^{-n\kappa(T-t)} J_T \Big),
\end{equation}
where $\succeq$ is the inequalities in the positive semidefinite ordering.
\end{proposition}

\begin{proof}
Combining the matrix Jensen inequality
\begin{equation}
n(t) = \mb{E}\big[ Q(t)^2 \big]\succeq  \mb{E}\big[-Q(t)\big]^2 = m(t)^2
\end{equation}
with \eqref{eq:for-einstein}, \eqref{eq:from-v-to-m} and expanding, we get \eqref{eq:matrix-ineq_opt}
\begin{equation} 
\begin{split}
\frac{\diff v(t)}{\diff t} \succeq v(t)^2- \Big( e^{-n\kappa(T-t)} J_T + \tfrac{c_T}{n} \cdot \msf{Id}_n \Big)& v(t)  - v(t) \Big( e^{-n\kappa(T-t)} J_T + \tfrac{c_T}{n}\cdot \msf{Id}_n \Big) 
\\ & +  \Big( e^{-n\kappa(T-t)} J_T + \tfrac{c_T}{n}\cdot \msf{Id}_n \Big)^2 + (n-1)\kappa v(t).
\end{split}
\end{equation}
The inequality \eqref{eq:matrix-ineq}  follows since the squared matrix is positive semidefinite. 
\end{proof}

Proposition \ref{prop:matrix-ineq} allows us to deduce the following local intrinsic dimensional logarithmic Sobolev inequalities which are, however, non-explicit.

\begin{theorem}
\label{thm:manifolds_opt}
Let $(\MM,\GG)$ be an $n$-dimensional Riemannian manifold with constant sectional curvature $\kappa\in\R$. Fix $T>0$, $x\in\MM$, a smooth positive function $f:\MM\to\R$ with $\int_M f\,\diff P_T\delta_x=1$, and let $\mu$ be the probability measure with $\frac{\diff\mu}{\diff P_T\delta_x}=f$. Suppose there is a family of matrices $U(t)\in \mathbb M_n(\R)$ for $t\in [0,T]$ which solves the equation
\begin{align} \label{eq:matrix-ODE}
\begin{split}
\frac{\diff U(t)}{\diff t} = U(t)^2- \Big( e^{-n\kappa(T-t)} J_T+\tfrac{c_T}{n}\cdot \msf{Id}_n  \Big) &U(t)  - U(t)\Big( e^{-n\kappa(T-t)} J_T+\tfrac{c_T}{n}\cdot \msf{Id}_n  \Big)\\ 
& +  \Big( e^{-n\kappa(T-t)} J_T + \tfrac{c_T}{n}\cdot \msf{Id}_n \Big)^2 + (n-1)\kappa U(t),
\end{split}
\end{align}
with either initial condition $U(0):=v(0)$ or $U(T):=v(T)$. Then, we have the local intrinsic dimensional logarithmic Sobolev inequality
\begin{equation} \label{eq:manifolds-lsi_opt}
\HH(\mu \| P_T\delta_x) \leq \ \frac{1}{2}\int_0^T \mr{tr}\big[U(t)\big] \,\diff t, \quad U(T)=\mb{E} \big( \nabla^{\mr{hor}} \log f(\Psi_T)\big)^{\otimes 2}.
\end{equation}
and the reverse local intrinsic dimensional logarithmic Sobolev inequality

\begin{equation} \label{eq:manifolds-rev-lsi_opt}
\HH(\mu \| P_T\delta_x) \geq \frac{1}{2}\int_0^T \mr{tr}\big[U(t)\big] \,\diff t, \quad U(0)=\big( \nabla^{\mr{hor}}\log P_Tf(x)\big)^{\otimes2}.
\end{equation}
\end{theorem}

\begin{proof}
Lehec's formula \eqref{eq:lehec-formula} implies
\begin{equation}  \label{eq:use-lehec_opt}
\HH(\mu \| P_T\delta_x)  \stackrel{\eqref{eq:lehec-formula}}{=} \frac{1}{2}\mb{E} \left[ \int_0^T \big| \nabla F_t(X_t)\big|^2\,\diff t \right] = \frac{1}{2} \sum_{i=1}^n \int_0^T \mb{E}\big[ \HH_i {\li F}_t(\Psi_t)^2\big]\,\diff t = \frac{1}{2}\int_0^T \mr{tr}\big[ v(t)\big] \,\diff t.
\end{equation}
For the reverse local intrinsic dimensional logarithmic Sobolev inequality, we note that $U(0)=v(0)$ so the result follows by \eqref{eq:matrix-ineq_opt} and standard comparison principles for matrix Ricatti equations, see \cite{Jon75}. For the local intrinsic dimensional logarithmic Sobolev inequality, we have $U(T)=v(T)$ and the conclusion follows by reversing time.
\end{proof}

Theorem \ref{thm:manifolds_opt} provides sharp results which are, however, not explicit since the solutions of \eqref{eq:matrix-ODE} are complicated. They are expressed in terms of special functions, except in the flat space case where they simplify considerably-- see Section \ref{subsec:intrinsic_LIS_flat}. To avoid the complication of Theorem \ref{thm:manifolds_opt} we will use \eqref{eq:matrix-ineq}, rather than the stronger inequality \eqref{eq:matrix-ineq_opt}, which will lead to explicit bounds, namely Theorem \ref{thm:manifolds_sec}. To this end, we shall need the following technical lemma on matrix Bernoulli differential inequalities.

\begin{lemma} \label{lem:bernoulli}
Fix $T>\e>0$, $n\in\N$, $\gamma\in\R\setminus\{0\}$ and let $A,B\in \mathbb M_n(\R)$ be symmetric matrices with $AB=BA$. Consider $C(t) \eqdef  \tfrac{e^{\gamma t}}{\gamma}A+tB$, where $t\in\R$. For any positive definite matrix $V_\e\in\mathbb M_n(\R)$, if a continuous function $V:[\e,T]\to\mathbb M_n(\R)$ for which every $V(t)$ is a positive semi-definite matrix satisfies the ordinary differential inequality
\begin{equation} \label{eq:ODE}
\forall \ t\in(\e,T),\qquad \frac{\diff V(t)}{\diff t} \succeq V(t)^2 + \big(e^{\gamma t}A+B\big) V(t) + V(t) \big( e^{\gamma t}A+B\big)
\end{equation}
with boundary condition $V(\e)=V_\e$, then it also satisfies the matrix inequalities
\begin{equation} \label{eq:comp-with-T}
\forall \ t\in[\e,T], \qquad V(t) \preceq e^{C(t)-C(T)} \Big( \msf{Id} _n+ V(T) \int_t^{T} e^{2C(s)-2C(T)}\,\diff s \Big)^{-1} V(T) e^{C(t)-C(T)}.
\end{equation}
and
\begin{equation} \label{eq:comp-with-0}
\forall \ t\in[\e,T],\qquad V(t) \succeq e^{C(t)-C(\e)} \Big( \msf{Id}_n - V(\e) \int_\e^t e^{2C(s)-2C(\e)}\,\diff s\Big)^{-1} V(\e) e^{C(t)-C(\e)}.
\end{equation}
Moreover, the right-hand side of \eqref{eq:comp-with-0} is positive definite for every $t\in(\e,T)$.
\end{lemma}

\begin{proof}
Since $A$ and $B$ commute,  we have
\begin{equation} \label{eq:der-expo}
\frac{\diff}{\diff t}e^{C(t)} = \big( e^{\gamma t}A+B\big)e^{C(t)} = e^{C(t)}\big( e^{\gamma t}A+B\big).
\end{equation}
As $V_\e$ is positive definite, the same holds for $V(t)$ for $t$ near $\e$ so let $t_{\max}\in [\e,T]$ be the supremum over $t\in [\e,T]$ where $V(t)$ is positive definite. For $t\in (\e,t_{\max})$, multiplying \eqref{eq:ODE} by $V(t)^{-1}$ on both sides, we deduce that
\begin{equation} \label{eq:ODE-for-V-1}
\frac{\diff V(t)^{-1}}{\diff t} \preceq -\msf{Id}_n - V(t)^{-1} C'(t) - C'(t) V(t)^{-1},
\end{equation}
where $C'(t) \eqdef \frac{\diff C(t)}{\diff t}$. Therefore, we have
\begin{equation} \label{eq:neater-odi}
\frac{\diff }{\diff t} \big[ e^{C(t)} V(t)^{-1} e^{C(t)} \big] \stackrel{\eqref{eq:der-expo}}{=} e^{C(t)} \big( C'(t) V(t)^{-1} + \frac{\diff V(t)^{-1}}{\diff t} + V(t)^{-1} C'(t) \Big) e^{C(t)} \stackrel{\eqref{eq:ODE-for-V-1}}{\preceq} -e^{2C(t)},
\end{equation}
where in the last inequality we used that $C(t)$ is symmetric.  Integrating from $\e$ to $t$, we get
\begin{equation}
e^{C(t)}V(t)^{-1}e^{C(t)} - e^{C(\e)}V(\e)^{-1}e^{C(\e)} \preceq - \int_\e^t e^{2C(s)}\,\diff s
\end{equation}
which can be rearranged to give, for every $t\in [\e,t_{\max})$,
\begin{equation} \label{eq:upper-for-V-1}
V(t)^{-1} \preceq e^{-C(t)} \Big( e^{C(\e)} V(\e)^{-1} e^{C(\e)} - \int_\e^t e^{2C(s)}\,\diff s\Big) e^{-C(t)}.
\end{equation}
Since the right-hand side of \eqref{eq:upper-for-V-1} is finite for every $t\in [\e,T]$, we can take the limit $t\uparrow t_{\max}$ to conclude that $V(t_{\max})$ is positive definite, and hence $t_{\max}=T$. Since the function $A \mapsto A^{-1}$ is operator decreasing on positive definite matrices, this proves \eqref{eq:comp-with-0} after some simple algebraic manipulations.  Moreover, as a consequence of \eqref{eq:upper-for-V-1},  the right-hand side of \eqref{eq:comp-with-0} is indeed positive definite. Similarly,   integrating \eqref{eq:neater-odi} from $t$ to $T$ and rearranging gives
\begin{equation} \label{eq:gives-pd}
\begin{split}
V(t)^{-1} & \succeq e^{-C(t)} \Big( e^{C(T)} V(T)^{-1} e^{C(T)} + \int_t^{T} e^{2C(s)} \,\diff s\Big) e^{-C(t)}.
\end{split}
\end{equation}
However, since $V(t)^{-1}$ is positive definite for every $t\in [\e,T]$ this is equivalent to
\begin{equation}
V(t) \preceq e^{C(t)-C(T)} \Big( \msf{Id}_n + V(T) \int_t^{T} e^{2C(s)-2C(T)}\,\diff s \Big)^{-1} V(T) e^{C(t)-C(T)},
\end{equation} 
which concludes the proof of \eqref{eq:comp-with-T}.
\end{proof}

\begin{proof} [Proof of Theorem \ref{thm:manifolds_sec}]
Fix $T>0$,  $\e>0$, and $x\in\MM$. Let $f:\MM\to\R$ be a smooth positive function with $\int_\MM f\,\diff P_T\delta_x = 1$ and let $\mu$ be the probability measure on $\MM$ with $\tfrac{\diff\mu}{\diff P_T\delta_x}=f$.  Without loss of generality, we can perturb $f$ and assume that 
\begin{equation}
v_\e \eqdef \mb{E}\big[ \big( \nabla^{\mr{hor}} \log  P_{T-\e}{\li f}(\Psi_\e)\big)^{\otimes 2}\big] 
\end{equation} 
is a positive definite matrix. Following the terminology above, Lehec's formula \eqref{eq:lehec-formula} implies
\begin{equation}  \label{eq:use-lehec}
\HH(\mu \| P_T\delta_x)  \stackrel{\eqref{eq:lehec-formula}}{=} \frac{1}{2}\mb{E} \left[ \int_0^T \big| \nabla F_t(X_t)\big|^2\,\diff t \right] = \frac{1}{2} \sum_{i=1}^n \int_0^T \mb{E}\big[ \HH_i {\li F}_t(\Psi_t)^2\big]\,\diff t = \frac{1}{2}\int_0^T \mr{tr}\big[ v(t)\big] \,\diff t.
\end{equation}
Since $v(\e) = v_\e$ is a positive definite matrix, Proposition \ref{prop:matrix-ineq} and Lemma \ref{lem:bernoulli} give
\begin{equation} \label{eq:from-odi-22}
\forall \ t\in[\e,T],\qquad v(t) \preceq e^{C(t)-C(T)} \Big( \msf{Id}_n + v_T \int_t^{T} e^{2C(s)-2C(T)}\,\diff s \Big)^{-1} v_T e^{C(t)-C(T)}
\end{equation}
where $C(t) = \frac{e^{\gamma t}}{\gamma} A + tB$, for the matrices
\begin{equation}
\begin{cases}
A =- e^{-n\kappa T} J_T \\ B = \big(\tfrac{(n-1)\kappa}{2} - \tfrac{c_T}{n}\big) \cdot \msf{Id}_n. \\ \gamma = n\kappa
\end{cases}
\end{equation}
By the perturbation above, we have thus established the validity of \eqref{eq:from-odi-22} for an arbitrary smooth positive density $f$ and for any $\e>0$.  Since $v_T = \mb{E}_\mu \big( \nabla \log  f\big)^{\otimes 2}$, the logarithmic Sobolev inequality of Theorem \ref{thm:manifolds_sec} follows by combining \eqref{eq:use-lehec} and \eqref{eq:from-odi-22} with $\e\to0^+$. The reverse logarithmic Sobolev inequality follows by using \eqref{eq:comp-with-0} since $v_0 = \big(\nabla \log P_Tf(x)\big)^{\otimes 2}$.
\end{proof}

\subsection{Intrinsic dimensional local logarithmic Sobolev inequalities in flat spaces}
\label{subsec:intrinsic_LIS_flat}
Our next goal is to prove the intrinsic dimensional local logarithmic Sobolev inequalities in flat spaces, i.e., equations \eqref{eq:intrinisc_lsi_dim} and \eqref{eq:intrinisc_lsi_dim_reverse}. In contrast to the proof of Theorem \ref{thm:manifolds_sec}, which uses the weaker inequality \eqref{eq:matrix-ineq}, here we will use the stronger inequality \eqref{eq:matrix-ineq_opt} which in flat space has an explicit clear solution. 
\begin{theorem}
\label{thm:flat}
Fix $T>0$ and $x\in\R^n$. Let $f:\R^n\to\R$ be a smooth positive function with $\int_{\R^n} f\,\diff P_T\delta_x = 1$ and let $\mu$ be the probability measure on $\R^n$ with $\tfrac{\diff\mu}{\diff P_T\delta_x}=f$. Then, we have the local intrinsic dimensional logarithmic Sobolev inequality
\begin{align}
\label{eq:intrinisc_lsi_dim_thm}
\HH(\mu \| P_T\delta_x) \leq \frac{T}{2} \Delta P_Tf (x)+\frac{1}{2}P_Tf(x)\log\det\left(\msf{Id}_n -T\frac{P_T(f\nabla^2\log f(x))}{P_Tf(x)}\right),
\end{align}
and the reverse local intrinsic dimensional logarithmic Sobolev inequality
\begin{align}
\label{eq:intrinisc_lsi_dim_reverse_thm}
\begin{split}
\HH(\mu \| P_T\delta_x) &\geq \frac{T}{2}\Delta P_Tf(x) -\frac{1}{2}P_Tf(x)\log\det\left(\msf{Id}_n+T\nabla^2\log P_Tf(x)\right).
\end{split}
\end{align}
\end{theorem}

\begin{proof}
The inequality \eqref{eq:intrinisc_lsi_dim_thm} follows by setting $\frac{\diff \mu}{\diff \lambda_n}\eqdef \frac{fP_T\delta_x}{P_Tf(x)}$ in \eqref{eq:lsi**}. To prove \eqref{eq:intrinisc_lsi_dim_reverse_thm}, we may assume without loss of generality assume that $-\nabla^2\log P_Tf(x)$ is invertible. Set $U(0)\eqdef v(0)=(\nabla\log P_Tf(x))^{\otimes 2}$ and use  the normalization assumption $\int_{\R^n} f\,\diff P_T\delta_x =P_Tf(x)=1$ to conclude that 
\begin{equation}
U(0)-\nabla^2P_Tf(x)  = (\nabla\log P_Tf(x))^{\otimes 2} - \frac{\nabla^2P_Tf(x)}{P_Tf(x)} = -\nabla^2\log P_Tf(x)
\end{equation}
is invertible. In flat space, using $\kappa=0$ and $J_T+\tfrac{c_T}{n}\cdot \msf{Id}_n= P_T\nabla^2 f(x)$, equation \eqref{eq:matrix-ODE} becomes
\begin{align} \label{eq:matrix-ODE_flat}
\frac{\diff U(t)}{\diff t} = U(t)^2-P_T\nabla^2 f(x)U(t)  - U(t) P_T\nabla^2 f(x) +  (P_T\nabla^2 f(x))^2.
\end{align}
The solution of \eqref{eq:matrix-ODE_flat} can be verified to be 
\begin{align} \label{eq:matrix-ODE_flat_sol}
\forall \ t\in(0,T), \qquad U(t)= \big([U(0)-P_T\nabla^2 f(x)]^{-1}-t \big)^{-1}+P_T\nabla^2 f(x),
\end{align}
where we used Hamilton's matrix inequality \eqref{eq:Hamilton} (see also Theorem \ref{thm:Hamilton} below) to justify the invertibility of  $[U(0)-P_T\nabla^2 f(x)]^{-1}-t$. Applying \eqref{eq:manifolds-rev-lsi_opt} of Theorem \ref{thm:manifolds_opt}  yields
\begin{align}\label{eq:lower_bound}
\begin{split}
\HH(\mu \| P_T\delta_x) &\geq -\frac{1}{2}\int_0^T \mr{tr}\big[\big((\nabla^2\log P_Tf(x))^{-1}+t \big)^{-1}\big]\,\diff t+\frac{T}{2}P_T\Delta f(x),
\end{split}
\end{align}
where again we used normalization assumption  $P_Tf(x)=1$. To rewrite the right-hand side of \eqref{eq:lower_bound} let $\{\lambda_i\}_{i=1}^n$ stand for the eigenvalues of $\nabla^2\log P_Tf(x)$ so
\begin{align}\label{eq:trace_term}
\begin{split}
\int_0^T \mr{tr}\big[\big((\nabla^2\log P_Tf(x))^{-1}+t \big)^{-1}\big]\,\diff t&=\sum_{i=1}^n\int_0^T (\lambda_i^{-1}+t)^{-1} \,\diff t=\sum_{i=1}^n\log \left(\frac{\lambda_i^{-1}+T}{\lambda_i^{-1}}\right) \\
&=\log\det\left(\msf{Id}_n+T\nabla^2\log P_Tf(x)\right).
\end{split}
\end{align}
It follows that \eqref{eq:lower_bound} reads
\begin{align}\label{eq:lower_bound_explicit}
\HH(\mu \| P_T\delta_x) &\geq \frac{T}{2}\Delta P_Tf(x) -\frac{1}{2}\log\det\left(\msf{Id}_n+T\nabla^2\log P_Tf(x)\right),
\end{align}
so using again $P_Tf(x)=1$, \eqref{eq:lower_bound_explicit} is equivalent to
\begin{align*}
\HH(\mu \| P_T\delta_x) &\geq \frac{T}{2}\Delta P_Tf(x) -\frac{1}{2}P_Tf(x)\log\det\left(\msf{Id}_n+T\nabla^2\log P_Tf(x)\right)\qedhere
\end{align*}
\end{proof}

\noindent {\bf Semigroup vs. stochastic interpolation.} The idea of writing the relative entropy as an integral  of a \emph{gradient term} goes back to the beginning of the Bakry--\'Emery theory of functional inequalities (see \cite{BE85,Bak94} or \cite[Section~5.5]{BGL14}). Such gradient terms often satisfy differential inequalities \`a la Proposition \ref{prop:matrix-ineq} which allow for the use of comparison principles in the spirit of Lemma \ref{lem:bernoulli}. For instance, Lehec in \cite{Leh17} considered the scalar-valued function $\alpha:[0,T]\to \R$ given by
\begin{equation}
\forall \ t\in[0,T],\qquad \alpha(t) = \mb{E}\big[ |\nabla F_t(X_t)|^2 \big]
\end{equation}
and showed that
\begin{equation} \label{eq:lehec-odi}
\forall \ t\in(0,T), \qquad \alpha'(t) \geq \frac{1}{n} \big( \alpha(t)-c_T\big)^2+(n-1)\kappa \alpha(t) \geq \frac{\alpha(t)}{n}\big( \alpha(t) + n(n-1)\kappa - 2c_T\big).
\end{equation}
Applying a standard comparison principle to the latter inequality, he then derived a dimensional upper bound for the relative entropy, see \cite[Equation~(25)]{Leh17}. It is worth pointing out that \eqref{eq:lehec-odi} is also a consequence of \eqref{eq:matrix-ineq_opt} after taking traces and using the elementary inequality $\mr{tr}C^2 \geq \tfrac{(\mr{tr}C)^2}{n}$ which holds for all $n\times n$ positive semidefinite matrices $C$. A close inspection of the arguments of this section reveals that the logarithmic Sobolev inequality of Theorem \ref{thm:manifolds_sec} is a strengthening of Lehec's result for manifolds of constant sectional curvature.

By reasoning similar to \eqref{eq:change-measure}, for every $i\in\{1,\ldots,n\}$, we have
\begin{equation}
\mb{E}\big[ (\HH_i {\li F_t}(\Psi_t))^2 \big] = \mb{E} \Big[ \frac{(\HH_i {\li G}_t(\Psi_t))^2}{{\li G}_t(\Psi_t)^2}\Big] = \mb{E}\Big[ \frac{(\HH_i {\li G}_t(\Phi_t))^2}{{(\li G}_t(\Phi_t))^2} f(B_T)\Big] = \mb{E} \Big[ \frac{(\HH_i {\li G}_t(\Phi_t))^2}{{\li G}_t(\Phi_t)}\Big]
\end{equation}
and thus
\begin{equation}
\forall \ t\in(0,T),\qquad \alpha(t) = P_t\left[ \frac{|\nabla P_{T-t}f|^2}{P_{T-t}f} \right](x).
\end{equation}
This semigroup representation of $\alpha(t)$ was used by Bakry, Bolley and Gentil in \cite[p.~405]{BBL17} to give an independent proof of inequality \eqref{eq:lehec-odi} for semigroups satisfying the curvature dimension condition $\msf{CD}(\rho,n)$, where $\rho=(n-1)\kappa$ (observe that $\alpha(t)$ is denoted by $\Lambda'(t)$ in their paper). Their main result \cite[Theorem~2.2]{BBL17} improves upon \cite[Equation~(25)]{Leh17} as they did not disregard the nonnegative constant in the second inequality of \eqref{eq:lehec-odi} and thus get to apply a tighter comparison principle. One could implement a similar strategy in the matricial setting treated here, by replacing the matrix inequality \eqref{eq:matrix-ineq} with the stronger inequality \eqref{eq:matrix-ineq_opt} and solving the corresponding ordinary differential equation. However, solutions of this equation appear to be non-explicit, so we chose the comparatively simpler presentation of Theorem \ref{thm:manifolds_sec} for clarity of the exposition. If one were to implement this reasoning, it is clear from the proofs of this section that the resulting inequality would improve upon \cite[Theorem~2.2]{BBL17}.

\section{Nonpositively curved space forms}
\label{sec:nonpos}
The goal of this section is to prove  Theorem \ref{thm:Hamilton_intro} (section \ref{subsec:matrix_inq}) and Theorem \ref{thm:nge_curved_intro} (section \ref{subsec:neg_curvature}). We conclude with section \ref{subsec:discussion} which discusses some of the ideas behind our proofs.

\subsection{Matrix inequalities}
\label{subsec:matrix_inq}
The main result of this section is the following Hamilton-type matrix inequality, namely, Theorem \ref{thm:Hamilton_intro}.
\begin{theorem}
\label{thm:Hamilton}
Let $(\MM,\GG)$ be an $n$-dimensional Riemannian manifold with constant nonpositive sectional curvature $\kappa \le 0$. Then, for every $T\ge 0$, 
\begin{align}
\label{eq:Hamilton_matrix_inq_negcurved_lambda=0}
\begin{split}
&\text{if, either }\kappa=0, \text{ or }\kappa<0\text{ and }\frac{4}{n^2\kappa}\frac{\Delta P_Tf(x)}{P_Tf(x)}=1,\quad\text{then}\quad -\nabla^2\log P_Tf(x)	\preceq   \frac{1}{T}\msf{Id}_n\quad\forall x\in \MM ,
\end{split}
\end{align}
Further,
\begin{align}
\label{eq:Hamilton_matrix_inq_negcurved_lambda>0}
\begin{split}
&\text{if }\kappa<0\text{ and }\frac{4}{n^2\kappa}\frac{\Delta P_Tf(x)}{P_Tf(x)})>1,\\
&\text{then}\quad 
-\nabla^2\log P_Tf(x)	\preceq \frac{n\kappa}{2}\left\{\sqrt{\frac{4}{n^2\kappa}\frac{\Delta P_Tf(x)}{P_Tf(x)}-1}\cot\left(\frac{n\kappa T}{2}\sqrt{\frac{4}{n^2\kappa}\frac{\Delta P_Tf(x)}{P_Tf(x)}-1}\right)-1\right\}\msf{Id}_n.
\end{split}
\end{align}
\end{theorem}

\begin{remark}
\label{rem:BBG}
To put Theorem \ref{thm:Hamilton_intro} in a larger context, and also to shed light on the conditions regarding $\frac{4}{n^2\kappa}\frac{\Delta P_Tf(x)}{P_Tf(x)}$, let us recall the improved Li-Yau inequality of Bakry, Bolley, and Gentil \cite[Corollaries 2.3,2.4]{BBL17}: Let $(\MM,\GG)$ be an $n$-dimensional Riemannian manifold with lower bound $(n-1)\kappa$ on its Ricci curvature. Let $\{P_t\}_{t\ge 0}$ be the associated heat semigroup and let $f:\MM\to \R$ be a positive function. Then, for every $x\in\MM$ and every $T\ge 0$,
\begin{equation}
\label{eq:Xbound}
\frac{4}{n(n-1)\kappa}\frac{\Delta P_Tf(x)}{P_Tf(x)}<1+\frac{4\pi^2}{(n-1)^2\kappa^2T^2}
\end{equation}
and
\begin{align}
\label{eq:improved_Li-Yau}
&-\Delta \log P_Tf(x)<\\
&\begin{cases}
\frac{n(n-1)\kappa}{2}\left\{\sqrt{\frac{4}{n(n-1)\kappa}\frac{\Delta P_Tf(x)}{P_Tf(x)}-1}\coth\left(\frac{(n-1)\kappa T}{2}\sqrt{\frac{4}{n(n-1)\kappa}\frac{\Delta P_Tf(x)}{P_Tf(x)}-1}\right) -1\right\}\text{if }\frac{4}{n(n-1)\kappa}\frac{\Delta P_Tf(x)}{P_Tf(x)}\le 1\\
\frac{n(n-1)\kappa}{2}\left\{\sqrt{\frac{4}{n(n-1)\kappa}\frac{\Delta P_Tf(x)}{P_Tf(x)}-1}\cot\left(\frac{(n-1)\kappa T}{2}\sqrt{\frac{4}{n(n-1)\kappa}\frac{\Delta P_Tf(x)}{P_Tf(x)}-1}\right) -1\right\}\text{if }1\le \frac{4}{n(n-1)\kappa}\frac{\Delta P_Tf(x)}{P_Tf(x)}<1+\frac{4\pi^2}{(n-1)^2\kappa^2T^2}.
\end{cases}
\end{align}
Hence, in the regime $1+\frac{1}{n-1}\le \frac{4}{n(n-1)\kappa}\frac{\Delta P_Tf(x)}{P_Tf(x)}<1+\frac{4\pi^2}{(n-1)^2\kappa^2T^2}$ we are able to obtain in hyperbolic spaces a \emph{matrix} version of the improved Li-Yau inequality.
\end{remark}

\begin{proof}[Proof of Theorem \ref{thm:Hamilton}]
We start by showing  that 
\begin{equation}\label{eq:m(t)def}
m(t)\eqdef\mb{E}[-\nabla^2\log P_{T-t}f(X_t)]
\end{equation}
satisfies the following differential inequality.
\begin{lemma}
\label{lem:m(t)ODI}
When $\kappa\le 0$,
\begin{equation}\label{eq:m(t)DI}
\forall \ t\in [0,T],\qquad \frac{\diff m(t)}{\diff t}\succeq m(t)^2+n\kappa m(t)+\kappa \Delta P_Tf(x)\cdot \GG.
\end{equation}
\end{lemma}
\begin{proof}
Define
\begin{equation}
\label{eq:u}
u(t)\eqdef e^{-n\kappa(T-t)}J_T+\frac{c_T}{n}\cdot\GG
\end{equation}
where $(J_T)_{ij} \eqdef P_T\nabla^2 f(\Phi_0 e_i, \Phi_0 e_j)(x) - \frac{1}{n} P_T \Delta f(x)\cdot \delta_{ij}$ and $c_T\eqdef P_T\Delta f(x)$. By \eqref{eq:from-v-to-m},
\begin{equation}
\label{eq:mJ_T}
m(t) = v(t) - u(t),
\end{equation}
so, by  \eqref{eq:matrix-ineq_opt},
\begin{align}
\label{eq:mJ_T_partial_t}
\frac{\diff m(t) }{\diff t}&=\frac{\diff v(t) }{\diff t}-n\kappa e^{-n\kappa(T-t)}J_T=\frac{\diff v(t) }{\diff t}-n\kappa u(t)+\kappa c_T\cdot \GG\\
&\succeq v(t)^2- u(t) v(t)  - v(t) u(t)+u(t)^2  + (n-1)\kappa v(t)-n\kappa u(t)+\kappa c_T\cdot \GG\\
&=m(t)^2 + n\kappa m(t)-\kappa v(t)+\kappa c_T\cdot \GG\\
&\succeq m(t)^2 + n\kappa m(t)+\kappa c_T\cdot \GG
\end{align}
where the last inequality uses that $\kappa\le 0$ and that $v(t)\succeq 0$ (since it is a nonnegative sum of rank-one matrices).
\end{proof}

The following technical lemma on matrix differential inequalities will allow us to further control the matrix $m(t)$.
\begin{lemma}\label{lem:W(t)}
Fix $T>0$ and let $W(t)$ be a family of matrices for $t\in [0,T]$ satisfying the differential inequality 
\begin{equation}\label{eq:W(t)}
\forall \ t\in [0,T],\qquad \frac{\diff W(t)}{\diff t}\succeq W(t)^2+\alpha W(t)+\beta\cdot \msf{Id}_n
\end{equation}
for some constant $\alpha,\beta\in \R$. Fix $\theta\in S^{n-1}$ and let $\phi(t)\eqdef \langle W(t)\theta,\theta\rangle$ for $t\in [0,T]$. Then, fixing $\phi(0)=c$ we have
\begin{align}\label{eq:phi(T)_lem}
\forall \ t\in[0,T],\qquad \phi(t)\ge \xi_{\lambda}(t)-\frac{\alpha}{2}
\end{align}
where 
\begin{align}\label{eq:xi_lem}
\xi_{\lambda}(t)\eqdef
\begin{cases}
\sqrt{\lambda}\tan(\sqrt{\lambda}t+c_1) & \text{if }\lambda>0\\
-\frac{1}{t+c_2}& \text{if }\lambda=0\\
-\sqrt{-\lambda}\tanh(\sqrt{-\lambda}t+c_3)
 & \text{if }\lambda<0,
\end{cases}
\end{align}
with 
\begin{align}\label{eq:lambda_lem}
\lambda\eqdef\beta-\frac{\alpha^2}{4}, 
\end{align}
and
\begin{align}\label{eq:c_lem}
 c_1\eqdef\arctan\left(\frac{1}{\sqrt{\lambda}}\left(c+\frac{\alpha}{2}\right)\right),\quad c_2\eqdef-\frac{2}{2c+\alpha}, \quad c_3\eqdef\arctanh\left(-\frac{1}{\sqrt{-\lambda}}\left(c+\frac{\alpha}{2}\right)\right).
\end{align}

\end{lemma}
\begin{proof}
Since $W(t)$ satisfies  $\frac{\diff W(t)}{\diff t} \succeq W(t)^2+\alpha W(t)+\beta\cdot \msf{Id}_n$ for all $t\in [0,T]$, we get that
\begin{align}\label{eq:phiIDE_W}
\begin{split}
\forall \ t\in [0,T],\qquad\frac{\diff \phi(t)}{\diff t}&=\left\langle \frac{\diff W(t)}{\diff t}\theta,\theta\right\rangle\ge \langle W(t)^2\theta,\theta\rangle+\alpha  \langle W(t)\theta,\theta\rangle+\beta |\theta|^2\\
&\ge \langle W(t)\theta,\theta\rangle^2+\alpha  \langle W(t)\theta,\theta\rangle+\beta |\theta|^2=\phi(t)^2+\alpha\phi(t)+\beta.
\end{split}
\end{align}
Hence, $\phi$ satisfies the ordinary differential inequality
\begin{align}\label{eq:phiIDE}
\forall \ t\in [0,T],\qquad\frac{\diff \phi(t)}{\diff t}\ge \phi(t)^2+\alpha\phi(t)+\beta.
\end{align}

The solution of the ordinary differential equation
\begin{align}\label{eq:tauODE}
\forall \ t\in [0,T],\qquad\frac{\diff \sigma(t)}{\diff t}= \sigma(t)^2+\alpha\sigma(t)+\beta , \quad\sigma(0)=c
\end{align}
is
\[
\sigma(t)= \xi_{\lambda}(t)-\frac{\alpha}{2}
\]
where 
\begin{align}\label{eq:xi}
\xi_{\lambda}(t)\eqdef
\begin{cases}
\sqrt{\lambda}\tan(\sqrt{\lambda}t+c_1) & \text{if }\lambda>0\\
-\frac{1}{t+c_2}& \text{if }\lambda=0\\
-\sqrt{-\lambda}\tanh(\sqrt{-\lambda}t+c_3)
 & \text{if }\lambda<0,
\end{cases}
\end{align}
with 
\begin{align}\label{eq:lambda}
\lambda\eqdef\beta-\frac{\alpha^2}{4}, 
\end{align}
and
\begin{align}\label{eq:c}
\begin{split}\quad c_1\eqdef\arctan\left(\frac{1}{\sqrt{\lambda}}\left(c+\frac{\alpha}{2}\right)\right),\quad c_2\eqdef-\frac{2}{2c+\alpha}, \quad c_3\eqdef\arctanh\left(-\frac{1}{\sqrt{-\lambda}}\left(c+\frac{\alpha}{2}\right)\right).
\end{split}
\end{align}
Applying standard comparison theorems \cite{Petrovitch} we get that
\begin{equation*}
\forall \ t\in[0,T], \qquad\phi(t)\ge \xi_{\lambda}(t)-\frac{\alpha}{2}\qedhere
\end{equation*}
\end{proof}
We are now ready for the proof of Theorem \ref{thm:Hamilton}. Recall that Lemma \ref{lem:W(t)} showed that the matrix $m(t)$ satisfies \eqref{eq:W(t)} with $\alpha=n\kappa$ and $\beta=\kappa\Delta P_Tf(x)$. In the following we let $\phi(t)\eqdef\langle m(t)\theta,\theta\rangle$ for $\theta\in S^{n-1}$. Let us distinguish between the flat and negatively curved cases. 

When $\kappa=0$ we need to show 
\begin{equation}\label{eq:Hamilton_flat}
\phi(0)\le \frac{1}{T}
\end{equation}
since we can choose $\theta$ to be any normalized eigenvector of $-\nabla^2\log P_Tf(x)$. If $\phi(0)\le 0$ then \eqref{eq:Hamilton_flat} is trivial so we may assume from now on that $\phi(0)>0$. As $\kappa=0$ we have $\alpha=\beta=0$ so $\lambda=\beta-\frac{\alpha^2}{4}=0$. Hence,  applying \eqref{eq:phi(T)_lem} we see that
\begin{align}
\label{eq:phi(t)-phi(0)}
\forall \ t\in[0,T],\qquad \phi(t)\ge \frac{\phi(0)}{1-\phi(0)t}.
\end{align}
In particular, \eqref{eq:phi(t)-phi(0)} implies that the denominator $1-\phi(0)t$ never vanishes since, otherwise, the right-hand side of \eqref{eq:phi(t)-phi(0)} is $+\infty$ (as $\phi(0)>0$) while the left-hand side is finite (as $\phi(t)<+\infty$). The non-vanishing of $1-\phi(0)t$, together with $\phi(0)>0$, implies that  
\begin{equation}\label{eq:phi(t)inq1}
0<1-t\phi(0)\quad \Longleftrightarrow \quad \phi(0)\le \frac{1}{t}.
\end{equation}
Taking $t=T$ establishes \eqref{eq:Hamilton_flat}.

When $\kappa<0$ we have $\alpha=\kappa c_T$ and $\beta=n\kappa$ so $\lambda=\kappa \Delta P_Tf(x)-\frac{n^2\kappa^2}{4}$. If $\frac{4}{n^2k}\Delta P_Tf(x)=1$, so that $\lambda=0$, then the argument proceeds as in the case $\kappa=0$. When $\frac{4}{n^2k}\Delta P_Tf(x)<1$, so that $\lambda>0$, applying \eqref{eq:phi(T)_lem} yields
\begin{align}\label{eq:phi(T)_neg}
\forall \ t\in[0,T],\qquad \phi(t)\ge \sqrt{\lambda}\tan(\sqrt{\lambda}t+c) -\frac{n\kappa}{2}
\end{align}
where
\begin{align}\label{eq:c_arctan}
c\eqdef\arctan\left(\frac{1}{\sqrt{\lambda}}\left(\phi(0)+\frac{n\kappa}{2}\right)\right).
\end{align}
In particular, as $\phi(t)$ is finite, it follows that $\tan(\sqrt{\lambda}t+c) <+\infty$ for every $t\in [0,T]$. At $t=0$, $c\in \left(-\frac{\pi}{2},\frac{\pi}{2}\right)$ by \eqref{eq:c_arctan} and the range of $\arctan$. It follows that, for every $t\in [0,T]$,
\begin{align}
\label{eq:c_bound}
\frac{\pi}{2}>\sqrt{\lambda}t+c=\sqrt{\lambda}t+\arctan\left(\frac{1}{\sqrt{\lambda}}\left(\phi(0)+\frac{n\kappa}{2}\right)\right).
\end{align}
Plugging in $t=T$ into \eqref{eq:c_bound}  and rearranging yields
\begin{align}
\label{eq:phi(0)_bound}
\phi(0)<\sqrt{\lambda}\tan\left(\frac{\pi}{2}-\sqrt{\lambda}T\right)-\frac{n\kappa}{2}=\sqrt{\lambda}\cot\left(\sqrt{\lambda}T\right)-\frac{n\kappa}{2}.
\end{align}
Letting $\theta$ to be any normalized eigenvector of $-\nabla^2\log P_Tf(x)$ and recalling the definition of $\lambda$ yields \eqref{eq:Hamilton_matrix_inq_negcurved_lambda>0} upon rearrangement.
\end{proof}

\begin{remark}
We cannot address in Theorem \ref{thm:Hamilton} the regime $\frac{4}{n^2\kappa}\frac{\Delta P_Tf(x)}{P_Tf(x)}<1$ since the $\tanh$ function, which will replace in the proof the $\tan$ function, is well-defined everywhere. It remains to be seen whether this is an artifact of the proof or an inherent obstacle. 
\end{remark}

\begin{remark}
The proof of Theorem \ref{thm:Hamilton} was obtained by using the inequality \eqref{eq:phi(T)_lem} subject to fixing the value of $\phi(0)$.  Analogously, we can fix the value of $\phi(T)$ and use  \eqref{eq:phi(T)_lem} (by reversing the time) to get different matrix inequalities. Indeed, our proof of Theorem \ref{thm:nge_curved} below makes use of the freedom to choose the initial (terminal) condition.
\end{remark}

\subsection{Intrinsic dimensional local logarithmic Sobolev inequalities in hyperbolic spaces}
\label{subsec:neg_curvature}
In this section we prove local intrinsic dimensional logarithmic Sobolev inequalities for the hyperbolic space, namely, Theorem \ref{thm:nge_curved_intro}. The inequalities provided by Theorem \ref{thm:nge_curved} below will be obtained as a consequence of the differential inequality of Lemma \ref{lem:m(t)ODI} together with the following simple observation:
\begin{lemma}
\label{lem:mteqivvt}
Let $f:\MM\to\R$ be such that $\int_M f\,\diff P_T\delta_x=1$
and let $\mu$ be the probability measure with $\frac{\diff\mu}{\diff P_T\delta_x}=f$.  Then,
\begin{align}
\label{eq:mteqivvt}
P_T(f\log f)(x)-P_Tf(x)\log P_Tf(x)=H(\mu|P_T\delta_x)=\frac{ P_T\Delta f(x)}{2}+\frac{1}{2}\int_0^T\mr{tr}[m(t)]\diff t.
\end{align}
\end{lemma}
\begin{proof}
By Theorem \ref{thm:Lehec}, \eqref{eq:vij}, and \eqref{eq:mJ_T},
\begin{equation}
\label{eq:Hmv}
H(\mu|P_T\delta_x)=\frac{1}{2}\int_0^T\mr{tr}[v(t)]\diff t=\frac{1}{2}\int_0^T\mr{tr}[m(t)]\diff t+\frac{1}{2}\int_0^T\mr{tr}[u(t)]\diff t
\end{equation}
where we recall \eqref{eq:u},
\begin{equation}
\label{eq:u1}
u(t)\eqdef e^{-n\kappa(T-t)}J_T+\frac{P_T\Delta f(x)}{n}\cdot\GG
\end{equation}
with $(J_T)_{ij} \eqdef P_T\nabla^2 f(\Phi_0 e_i, \Phi_0 e_j)(x) - \frac{1}{n} P_T \Delta f(x)\cdot \delta_{ij}$. The proof is complete since $\mr{tr}[J_T]=0$.
\end{proof}

\begin{theorem}
\label{thm:nge_curved}
Let $(\MM,\GG)$ be the $n$-dimensional hyperbolic space with sectional curvature $\kappa<0$ with the associated  heat semigroup $\{P_t\}_{t\ge 0}$. Fix $T>0$, $x\in\MM$, a smooth positive function $f:\MM\to\R$ with $\int_M f\,\diff P_T\delta_x=1$,
and let $\mu$ be the probability measure with $\frac{\diff\mu}{\diff P_T\delta_x}=f$. 
Then, with
\begin{align}
\lambda\eqdef \frac{n^2\kappa^2}{4}\left\{\frac{4}{n^2\kappa}\Delta P_Tf(x)-1\right\}, \quad 
\alpha_i\eqdef
\begin{cases}
\arctan\left(\frac{1}{\sqrt{\lambda}}\left(\sigma_i+\frac{n\kappa}{2}\right)\right)  &\text{if }\lambda>0,\\
-\frac{2}{2\sigma_i+n\kappa}  &\text{if }\lambda=0,\\
\arctanh\left(-\frac{1}{\sqrt{-\lambda}}\left(\sigma_i+\frac{n\kappa}{2}\right)\right)&\text{if }\lambda<0,\\
\end{cases}
\end{align}
we have the local intrinsic dimensional logarithmic Sobolev inequality
\begin{align} \label{eq:manifolds-lsi_nonpos}
\begin{split}
P_T(f\log f)(x)-P_Tf(x)\log P_Tf(x)\leq \frac{ P_T\Delta f(x)}{2}-\frac{n^2\kappa T}{2}-\frac{1}{2}
\begin{cases}
\sum_{i=1}^n\log\left(\frac{\cos (\alpha_i)}{\cos(\sqrt{\lambda}T+\alpha_i)}\right)&\text{if }\lambda>0\\
\sum_{i=1}^n\log\left(\frac{\alpha_i}{T+\alpha_i}\right) &\text{if }\lambda=0\\
\sum_{i=1}^n\log\left(\frac{\cosh (\alpha_i)}{\cosh(\sqrt{-\lambda}T+\alpha_i)}\right) &\text{if }\lambda<0\\
\end{cases}
\end{split}
\end{align}
where $\{\sigma_i\}_{i=1}^n$ are the eigenvalues of $\mb{E}_{\mu}[-\nabla^2\log f]$, and the reverse local intrinsic dimensional logarithmic Sobolev inequality
\begin{align} \label{eq:manifolds-rev-lsi_nonpos}
\begin{split}
P_T(f\log f)(x)-P_Tf(x)\log P_Tf(x)\geq \frac{ P_T\Delta f(x)}{2}-\frac{n^2\kappa T}{2}+\frac{1}{2}
\begin{cases}
\sum_{i=1}^n\log\left(\frac{\cos (\alpha_i)}{\cos(\sqrt{\lambda}T+\alpha_i)}\right)&\text{if }\lambda>0\\
\sum_{i=1}^n\log\left(\frac{\alpha_i}{T+\alpha_i}\right) &\text{if }\lambda=0\\
\sum_{i=1}^n\log\left(\frac{\cosh (\alpha_i)}{\cosh(\sqrt{-\lambda}T+\alpha_i)}\right) &\text{if }\lambda<0\\
\end{cases}
\end{split}
\end{align}
where $\{\sigma_i\}_{i=1}^n$ are the eigenvalues of $-\nabla^2\log P_Tf(x)$.
\end{theorem}
\begin{proof}
Given any basis $\{\theta_i\}_{i=1}^n$ of $\R^n$ we have
\begin{align}
\label{eq:mphi}
\forall \ t\in [0,T],\qquad \mr{tr}[m(t)]=\sum_{i=1}^n\phi_i(t)
\end{align}
where $\phi_i(t)\eqdef\langle m(t)\theta_i,\theta_i\rangle$ for $i=1,\ldots, n$. It follows from Lemma \ref{lem:m(t)ODI} and Lemma \ref{lem:W(t)} that
\begin{align}
\label{eq:mphixi}
\forall \ t\in [0,T],\qquad \mr{tr}[m(t)]\ge \left(\sum_{i=1}^n\xi_{i,\lambda}(t)\right)-\frac{n^2\kappa}{2}
\end{align} 
where 
\begin{align}\label{eq:xi_proof}
\xi_{i,\lambda}(t)\eqdef
\begin{cases}
\sqrt{\lambda}\tan(\sqrt{\lambda}t+c_{i,1}) & \text{if }\lambda>0\\
-\frac{1}{t+c_{i,2}}& \text{if }\lambda=0\\
-\sqrt{-\lambda}\tanh(\sqrt{-\lambda}t+c_{i,3})
 & \text{if }\lambda<0,
\end{cases}
\end{align}
with 
\begin{align}\label{eq:lambda_proof}
\lambda\eqdef\kappa \Delta P_Tf(x)-\frac{n^2\kappa^2}{4}, 
\end{align}
and
\begin{align}\label{eq:c_proof}
c_{i,1}\eqdef\arctan\left(\frac{1}{\sqrt{\lambda}}\left(\phi_i(0)+\frac{n\kappa}{2}\right)\right),\quad c_{i,2}\eqdef-\frac{2}{2\phi_i(0)+n\kappa}, \quad c_{i,3}\eqdef\arctanh\left(-\frac{1}{\sqrt{-\lambda}}\left(\phi_i(0)+\frac{n\kappa}{2}\right)\right).
\end{align}
It follows from Lemma \ref{lem:mteqivvt} that
\begin{equation}
P_T(f\log f)(x)-P_Tf(x)\log P_Tf(x)\ge \frac{ P_T\Delta f(x)}{2}-\frac{n^2\kappa T}{2}+\frac{1}{2}\sum_{i=1}^n\int_0^T  \xi_{i,\lambda}(t)\diff t.
\end{equation}
Hence, taking $\{\theta_i\}_{i=1}^n$ to be the eigenvectors of $m(0)=-\nabla^2\log P_Tf(x)$, and integrating $\{\xi_{i,\lambda}(t)\}$, yields \eqref{eq:manifolds-rev-lsi_nonpos}.

To prove \eqref{eq:manifolds-lsi_nonpos} we define $\tilde\phi_i(t):=\phi_i(T-t)$ which satisfies 
\begin{align}
\forall \ t\in[0,T],\qquad \tilde\phi(t)\le \xi_{i,\lambda}(T-t)-\frac{\alpha}{2}
\end{align}
where now 
\begin{align}
\begin{split}\quad c_{i,1}\eqdef\arctan\left(\frac{1}{\sqrt{\lambda}}\left(\phi_i(T)+\frac{n\kappa}{2}\right)\right),\quad c_{i,2}\eqdef-\frac{2}{2\phi_i(T)+n\kappa}, \quad c_{i,3}\eqdef\arctanh\left(-\frac{1}{\sqrt{-\lambda}}\left(\phi_i(T)+\frac{n\kappa}{2}\right)\right).
\end{split}
\end{align}
The proof now proceeds as in the proof of \eqref{eq:manifolds-rev-lsi_nonpos}.
\end{proof}

\subsection{Discussion}
\label{subsec:discussion}
We conclude this section by discussing the roles of matrix differential inequalities in our proofs. 
\newline

\noindent {\bf Matrix differential inequalities.} The master matrix differential inequality \eqref{eq:matrix-ineq_opt}, which is at the core of all of our proofs, can be expressed either in terms of $v(t)$,
\begin{align}
\label{eq:v_appen}
\frac{\diff v(t)}{\diff t} \succeq v(t)^2- u(t)v(t)  - v(t) u(t) + u(t)^2 + (n-1)\kappa v(t),
\end{align}
or in terms of $m(t)$,
\begin{align}
\label{eq:mfull_appen}
\frac{\diff m(t)}{\diff t} \succeq  m(t)^2 + n\kappa m(t)-\kappa v(t)+\kappa c_T\cdot \GG. 
\end{align}
The inequalities \eqref{eq:v_appen} and \eqref{eq:mfull_appen} are equivalent and contain the same information. In particular, in  \emph{flat} space forms, where $\kappa=0$, both inequalities are of the form $\frac{\diff W(t)}{\diff t} \succeq  W(t)^2$. In curved spaces, there are two different ways to proceed from \eqref{eq:v_appen} and \eqref{eq:mfull_appen}:

\begin{enumerate}
\item Omit the term $u(t)^2$ from \eqref{eq:v_appen} to get
\begin{align}
v(t)\succeq U(t)
\end{align}
where 
\begin{align}
\label{eq:Udiscuss_v}
\frac{\diff U(t)}{\diff t} = U(t)^2- u(t)U(t)  - U(t) u(t) + (n-1)\kappa U(t).
\end{align}
The point of omitting $u^2(t)$ is that equation \eqref{eq:Udiscuss_v} can be solved explicitly, in contrast to the equation resulting if we keep the $u^2(t)$ term.\footnote{If we take the trace in \eqref{eq:v_appen} then the equation with the $u(t)^2$ term can be solved explicitly---see the end of section \ref{sec:LSImanifold}. However, if we do so we would get the ambient dimension $n$ rather than the intrinsic dimension.}
\item Omit the term  $-\kappa v(t)$ from  \eqref{eq:mfull_appen}, which can be done only in  \emph{negatively} curved space forms
to get 
\begin{align}
m(t)\succeq U(t)
\end{align}
where
\begin{align}
\label{eq:Udiscuss_m}
\frac{\diff U(t)}{\diff t} = U(t)^2+n\kappa U(t)+\kappa c_T\cdot \GG.
\end{align}
Again, the point of omitting $-\kappa v(t)$ is so that \eqref{eq:Udiscuss_m} can be solved explicitly. Note that in flat spaces, there is no loss in omitting $-\kappa v(t)$.
\end{enumerate}

\noindent {\bf Matrix vs. trace differential inequalities.}
The proofs of Theorem \ref{thm:nge_curved} and Theorem \ref{thm:manifolds_sec} proceed along similar but different lines. Both proofs start by establishing an inequality of the form 
\begin{align}
\frac{\diff W(t)}{\diff t} \succeq \mathcal F(W(t))
\end{align}
for some quadratic functional $\mathcal F$. The goal is to bound $\mr{tr}[W(t)]$ which can be achieve by two means. Letting $\{U(t)\}$ be the solution to 
\begin{align}
\frac{\diff U(t)}{\diff t} = \mathcal F(U(t))
\end{align}
we could:
\begin{enumerate}
\item Argue that $W(t)\succeq U(t)$ and then take the trace on both sides to get 
\begin{align}
\mr{tr}[W(t)]\ge\mr{tr}[U(t)].
\end{align}
This is the method used to prove Theorem \ref{thm:manifolds_sec}  and  Theorem \ref{thm:flat} (with different functionals $\mathcal F$).
\item  When $\mathcal F$ has scalar (rather than matrix) coefficients, it holds that 
\begin{equation}
\langle \mathcal F(W(t))\theta,\theta\rangle \ge  \mathcal F(\langle W(t)\theta,\theta\rangle)
\end{equation}
 with strict inequality unless $\theta$ is an eigenvector of $W(t)$. We can then $\{\theta_i\}$ to be any basis and let $\phi_{i,W}(t):=\langle W(t)\theta_i,\theta_i\rangle$, $\phi_{i,U}(t):=\langle U(t)\theta_i,\theta_i\rangle$ so
\begin{align}
\frac{\diff \phi_{i,W}(t)}{\diff t}\ge \mathcal F(\phi_{i,W}(t)), \quad \frac{\diff \phi_{i,U}(t)}{\diff t}=\mathcal F(\phi_{i,U}(t)),
\end{align}
which shows $\phi_{i,W}(t)\ge \phi_{i,U}(t)$. Hence, for any basis $\{\theta_i\}$ we have
\begin{align}
\mr{tr}[W(t)]=\sum_i\phi_{i,W}(t)\ge \sum_i\phi_{i,U}(t)= \mr{tr}[U(t)].
\end{align}
This is the method used to prove Theorem \ref{thm:nge_curved}.
\end{enumerate}
While both methods lead to the inequality 
\begin{align}
\mr{tr}[W(t)]\ge \mr{tr}[U(t)],
\end{align}
the second method is weaker since  the inequality $\frac{\diff \phi_{i,W}(t)}{\diff t}\ge \mathcal F(\phi_{i,W}(t))$ is weaker in principle than $\frac{\diff W(t)}{\diff t} \succeq \mathcal F(W(t))$ unless $\theta_i$ is an eigenvector of $W(t)$. However, for the purpose of proving an inequality for the \emph{trace}, there is no loss since the trace is invariant under rotations so for each $t$ we can introduce a rotation $R(t)$ which takes $\{\theta_i\}$ to the eigenvectors of $W(t)$ or $U(t)$.

\bibliographystyle{plain}
\bibliography{intrinsic_dim_V3}

\end{document}